\documentclass[reqno,12pt]{amsart}
\usepackage{epsfig,color}

\newtheorem{Thm}[equation]{Theorem}
\newtheorem{Cor}[equation]{Corollary}
\newtheorem{Lem}[equation]{Lemma}
\newtheorem{Pro}[equation]{Proposition}

\theoremstyle{definition}

\newtheorem{Def}[equation]{Definition}
\newtheorem{Exa}[equation]{Example}

\theoremstyle{remark}

\newtheorem{Rem}[equation]{Remark}

\numberwithin{equation}{section}
\makeatletter
\renewcommand{\c@figure}{\c@equation}
\makeatother

\newcommand{\maths}[1]{{\bf #1}}

\newcommand{\RR}{\maths{R}}

\newcommand{\ra}{\rightarrow}

\newcommand{\bord}{\partial}

\newcounter{fig}

\newcommand{\dem}{{\sl {Proof.} \  }}
\newcommand{\eop}[1]{{\flushright\hfill\fbox{\bf #1}}}

\newcommand{\ack}{\noindent{\bf Acknowledgement.}}

\begin{document}

\title[Dirichlet-Neumann Problem]{Boundary Value Problems on Planar Graphs and Flat Surfaces with integer cone singularities, II: The mixed Dirichlet-Neumann Problem}

\author{Sa'ar Hersonsky}

\address{Department of Mathematics\\ 
University of Georgia\\ 
Athens, GA 30602}

\urladdr{http://www.math.uga.edu/~saarh}
\email{saarh@math.uga.edu}
%\thanks{}
%\keywords{harmonic functions on graphs, flat surfaces, discrete uniformization}
%\subjclass[2000]{Primary: Geometric Topology; Secondary: subject}
\date{May 28; version 0.0528103}

\begin{abstract}
In this paper we continue the study started in \cite{Her1}. We consider a planar, bounded, $m$-connected region $\Omega$, and let $\bord\Omega$ be its boundary.  Let $\mathcal{T}$ be a cellular decomposition of $\Omega\cup\bord\Omega$, where each $2$-cell is either a triangle or a quadrilateral. From these data and a conductance function we construct a canonical pair $(S_{},f)$ where $S$ is a 
special type of a (possibly immersed) 
genus $(m-1)$ {\it singular flat surface},  tiled by rectangles and $f$ is an energy preserving mapping from ${\mathcal T}^{(1)}$ onto $S_{}$. 
  In \cite{Her1} the solution of a Dirichlet problem defined on ${\mathcal T}^{(0)}$ was utilized,  in this paper we  employ the solution of a mixed Dirichlet-Neumann problem.

  %By a singular flat surface, we will mean a surface which carries a metric structure locally modeled on the Euclidean plane, except at a finite number of points. These points have cone singularities, and the cone angle is allowed to take any positive value (see for instance \cite{Tro} for an excellent survey).  
%Our realization may be considered as a discrete uniformization of planar bounded regions. 
\end{abstract}

\maketitle

\section{Introduction}
\label{se:Intro}
Before stating our main result,  we need to define a special kind of two dimensional objects, {\it surfaces with propellors}. A flat, genus zero compact surface with $m> 2$ boundary components endowed with conical singularities, will be called  {\it a ladder of singular pairs of pants}. A {\it sliced} Euclidean rectangle is a Euclidean rectangle in which two adjacent vertices are identified, and possibly a finite number of points on the opposite edge have been pinched.
\begin{Def} 
 A  singular flat (possibly immersed), genus zero compact surface with $m> 2$ boundary components having conical singularities, will be called a surface with propellors, if it has a decomposition into the following pieces:
 a ladder of singular pairs of pants, sliced Euclidean rectangles, Euclidean rectangles and straight Euclidean cylinders. 
 \end{Def}

We consider (as in \cite{Her1}) a planar, bounded, $m$-connected region $\Omega$, and let $\bord\Omega$ be its boundary.  Let $\bord\Omega=E_1\sqcup E_2$, where $E_1$ is the outermost component of $\bord\Omega$. Henceforth, we will let $\{\alpha_1,\ldots, \alpha_l\}$ be a collection of closed disjoint arcs contained in $E_1$, and let 
 $\{\beta_1,\ldots, \beta_m\}$ be a collection of closed disjoint arcs contained in 
$E_2$.  Let $\mathcal{T}$ be a cellular decomposition of $\Omega\cup\bord\Omega$, where each $2$-cell is either a triangle or a quadrilateral.  Invoke a {\it conductance function} on ${\mathcal T^{(1)}}$, making it a {\it finite network}, and use it to define a combinatorial Laplacian $\Delta$ on ${\mathcal T}^{(0)}$. 
These data will be called {\it Dirichlet-Neumann data} for $(\Omega,\bord\Omega, {\mathcal T})$.  Let $k$ be a positive constant, and  let $g$ be the solution of a mixed {\it Dirichlet-Neumann} boundary value problem (DN-BVP) defined on ${\mathcal T}^{(0)}$ and determined by requiring that 
\begin{enumerate}
\item $g|_{\alpha_i}=k,\ \mbox{\rm for all}\  i=1,\ldots, l$, and $g|_{E_2\setminus (\beta_1\cup\ldots\cup \beta_m)}=0$,
\item $\dfrac{\bord g}{\bord n}|_{\beta_j}=0,\ \mbox{\rm for all}\  j=1,\ldots, m$,
\item $\Delta g=0$ at every interior vertex of ${\mathcal T}^{(0)}$, i.e. $g$ is {\it combinatorially  harmonic}, and
\item $\sum_{x\in\partial\Omega}\frac{\bord g}{\bord n}(\partial \Omega)(x)=0,$
\end{enumerate}
 where ($4$) is a necessary consistent condition.
Let $E(g)$ denote the {\it Dirichlet energy} of $g$.  We may now state the main result of this paper:
\begin{Thm} {\rm (Main result)}
  \label{Th:ladder}
 Let  $(\Omega,\bord\Omega,{\mathcal T})$ be a bounded, $m$-connected, planar region endowed with a Dirichlet-Neumann data, with $m>2$. % Let $S_{\Omega}$ be a ladder of singular pairs of pants such that
% \begin{enumerate}
 %\item ${\mbox{\rm Length}_{\rm Euclidean}}(S_{\Omega})_{E_1}= \ \ \ \sum_{x\in E_1}\frac{\bord g}{\bord n}(x)$, and
%  \item ${\mbox{\rm Length}_{\rm Euclidean}}(S_{\Omega})_{E_2^i}=- \sum_{x\in E_2^i}\frac{\bord g}{\bord n}(x)$, for $i=1,\cdots, m-1$,
% \end{enumerate}
% where $(S_{\Omega})_{E_1}$ and $(S_{\Omega})_{E_2^i}$, for $i=1,\cdots, m-1$, are the boundary components of $S_{\Omega}$.
%   Then, there exists
   %\begin{enumerate}
   %\item a cellular decomposition $\hat{\mathcal  T}$ (comprised of triangles and quadrilaterals), 
  %which is a refinement of  ${\mathcal T}$, and is determined by the critical level sets of $g$, as well as
   %\item 
Then there exists a surface with propellors  ${\mathcal S}_{\Omega}$, and a mapping $f$ which associates to each edge in  ${\mathcal  T}^{(1)}$ a unique  Euclidean rectangle in $S_{\Omega}$ in such a way that the collection of these rectangles forms a tiling of $S_{\Omega}$. 
  % \end{enumerate}
   Furthermore, $f$ is  boundary preserving, and $f$ is energy preserving in the sense that  $E(g)= {\rm Area}(S_{\Omega})$. 
\end{Thm}
Throughout this paper a Euclidean rectangle will denote the image under an isometry of a planar Euclidean rectangle. For instance, some of the image rectangles that we will construct embed in a flat Euclidean cylinder. These will further be glued in a way that will not distort the Euclidean structure (see  
\S\ref{se:Qaud} and \S\ref{se:low complexity} for the details).

In our setting, boundary preserving means that the rectangle associated to an edge $[u,v]$ in ${\mathcal  T}^{(1)}$ with $u\in \partial \Omega$ has one of its edges on a corresponding boundary component of the surface. 
In the course of the proof of Theorem~\ref{Th:ladder}, it will become apparent that the number of singular points and their cone angles, the lengths of shortest geodesics between boundary curves in the ladder and the number of propellors,  may be explicitly determined. In particular, the cone angles obtained by our construction are always $\pi/2$, or even multiples of $\pi$. Some classes of such surfaces are called {\it translation surfaces}, and for excellent accounts see for instance  \cite{HuMaScZo},  \cite{Ma}  and \cite{Zo}.  

Also, the dimensions of each rectangle are determined by the given DN-BVP problem on 
${\mathcal  T}^{(0)}$.  Concretely, for $[u,v]\in {\mathcal  T}^{(1)}$, the associated rectangle will have its height equals to $(g(u)-g(v))$ and its width equals to $c(u,v)(g(u)-g(v))$, when $g(u)>g(v)$. Some of the rectangles are not embedded.  We will comment on this point (which is also transparent in the proof of the theorem above) in Remark \ref{re:embedd}. In a snapshot, some of the rectangles which arise from intersection of edges with singular level curves of the DN-BVP solution are not  embedded.
\medskip

A {\it surface with dents and pillows} will denote the surface obtained by doubling a surface with propellors along its boundary.
The following Corollary is straightforward, thus establishing the statement in the abstract of this paper. 
 \begin{Cor} 
 \label{Cor:surface}
 Under the assumptions of Theorem~\ref{Th:ladder}, there exists a canonical pair $(S,f)$, where $S$ is a flat surface with dents and pillows of genus $(m-1)$, having conical singularities and tiled by Euclidean rectangles.  The mapping $f$ is energy preserving  from
 ${\mathcal T}^{(1)}$ into $S$, in the sense that $2E(g)={\rm Area}(S)$.   %Moreover, $S$ admits a pair of pants decomposition whose dividing curves have Euclidean lengths given by ${\rm (1)-(3)}$ of Theorem~\ref{Th:ladder}.
 
 \end{Cor} 
 \begin{proof}
 
 Given $(\Omega,\bord\Omega,{\mathcal T})$, glue together two copies of $S_{\Omega}$ (their existence is guaranteed by Theorem~\ref{Th:ladder}) along corresponding boundary components.  This results in a flat surface $S=S_{\Omega}\bigcup\limits_{ \bord\Omega }S_{\Omega}$ of genus $(m-1)$ and a mapping ${\bar f}$ which restricts to $f$ on each copy.
 
\end{proof}
 
%A wide range of applications is derived once a surface is tiled by squares. For example, consult \cite{Hi} for spectral theory questions and \cite{Zo1} for volume computations of strata in the moduli space of abelian differentials on a Riemann surface as well as  \cite{EsOb} and \cite{KoZo}. 
%Using the results of this paper and its sequel (\cite{Her2}), some of these applications will be extended, and other applications will be given (\cite{Her3}).

\medskip
The following two theorems are foundational and serve as  building blocks in the proofs of the theorems above. In \cite[Theorem 0.4]{Her1} we proved the following: 
\begin{Thm}
 {\rm (Discrete uniformization of an annulus \cite{BSST})} 
\label{Th:annulus} 
Let  ${\mathcal A}$ be an annulus and let $(\Omega,\bord\Omega,{\mathcal T})= ({\mathcal A},\bord{\mathcal A}=E_1\sqcup E_2,{\mathcal T})$.  Let $S_{\mathcal A}$ be a straight Euclidean cylinder with height $H=k$  and 
circumference 
\begin{equation}
\label{eq:circumference}
C=\sum_{x\in E_1}\frac{\bord g}{\bord n}(x).
\end{equation}
% (The right hand side  should be thought of as the discrete {\it flux} of $g$ through $E_1$).
Then 
there exists a mapping $f$ which associates to each edge in  ${\mathcal T}^{(1)}$ a unique embedded Euclidean rectangle in $S_{\mathcal A}$ in such a way that the collection of these rectangles forms a tiling of  $S_{\mathcal A}$. Furthermore, $f$ is boundary preserving, and $f$ is energy preserving in the sense that $E(g)= {\rm Area}(S_{\mathcal A})$.
\end{Thm}
 
 A topological planar closed disk with four
distinguished points 
%$P_1, P_2,
%P_3. P_4$
on its boundary, its corners, will be called a quadrilateral. %The following definition is due to Schramm (\cite{Sch}).
%\begin{Def}
%\label{de:triangulated quad} 
Let ${\mathcal R}$ be a 
quadrilateral endowed with a cellular decomposition. Let $\bord {\mathcal R}=
\partial{\mathcal R}_{\mbox{\tiny\rm bottom} } \cup \partial{\mathcal R}_{\mbox{\tiny\rm left} } \cup \partial{\mathcal R}_{\mbox{\tiny\rm top} }
\cup \partial{\mathcal R}_{\mbox{\tiny\rm right} }$ be a decomposition of $\bord {\mathcal R}$ into four 
non-trivial arcs of the cellular decomposition with disjoint interiors, in cyclic order. If the intersection of any two of these arcs is not empty, then it consists of a corner (all of which are vertices). A corner belongs to one and
only one of the arcs.  We solve a mixed Dirichlet-Neumann boundary value problem with  $E_1=E_2$, $\alpha_1=\partial{\mathcal R}_{\mbox{\tiny\rm right} }, \beta_1=  \partial{\mathcal R}_{\mbox{\tiny\rm bottom} }, \beta_2= \partial{\mathcal R}_{\mbox{\tiny\rm top} }$ (and hence that  and $g|_{\partial{\mathcal R}_{\mbox{\tiny\rm left} }}=0$). The second foundational theorem is proved in this paper:

\begin{Thm}  {\rm (Discrete uniformization of a rectangle \cite{BSST})}
\label{Th:rectangle}
Let ${\mathcal R}$ be a quadrilateral. Let $S_{{\mathcal R}}$ be a Euclidean rectangle with width $W=k$ and  height 
\begin{equation}
\label{eq:height}
H=\sum_{x\in\partial{\mathcal R}_{\mbox{\tiny\rm right}}}\frac{\bord g}{\bord n}(\mathcal R)(x).
\end{equation}
Then there exists a mapping $f$ which associates to each edge in  ${\mathcal T}^{(1)}$ a unique embedded Euclidean rectangle in $S_{\mathcal R}$ in such a way that the collection of these rectangles forms a tiling of  $S_{\mathcal R}$. Furthermore, $f$ is boundary preserving, and $f$ is energy preserving in the sense that $E(g)= {\rm Area}(S_{\mathcal R})$.

\end{Thm}

Given $(\Omega,\bord\Omega,{\mathcal T})$, we will work with the natural affine structure induced by the cellular decomposition. Let us denote this cell complex endowed with its affine structure by ${\mathcal C}(\Omega,\partial \Omega,{\mathcal T})$. 
 %There is a common thread in our proofs of the theorems above. Extend $g$ to an affine map ${\bar g}$ defined on ${\mathcal T}$. This results in a piecewise linear structure on ${\mathcal T}$.
  Next, study the level curves of ${\bar g}$ on a $2$-dimensional complex which is homotopically equivalent  to ${\mathcal C}(\Omega,\partial \Omega,{\mathcal T})$, embedded in $\RR^{3}$, obtained by using ${\bar g}$ as a height function on $(\Omega,\bord\Omega,{\mathcal T})$. We will work with the level curves of  ${\bar g}$ or equivalently, with their projection on ${\mathcal C}(\Omega,\partial \Omega,{\mathcal T})$.
  
Once Theorem~\ref{Th:rectangle} is proved, we proceed to prove Theorem~\ref{Th:ladder} as follows.
We first construct a topological decomposition of $\Omega$.
 For $i=0,\ldots, k$, consider the sub-domain of $\Omega$ defined by 
\begin{equation}
\label{eq:suba}
\Omega_i=\{x\in\Omega\  |\  p_i<g(x)<p_{i+1}\},
\end{equation}
(where the value at $x$ which is not a vertex is defined by the affine extension of $g$). In general $\Omega_i$ is multi-connected and (by definition) contains no singular vertices in its interior. Let $g_i=g|_{\Omega_i}$ be the restriction of $g$ to $\Omega_i\cup\partial  \Omega_i$. The definition of $g_i$ involves (as in the proof of Theorem~\ref{Th:rectangle}) new vertices (type I), and new edges and their conductance constants.  In particular,  each $g_i$ is the solution of a D-BVP (Dirichlet boundary value problem) or a DN-BVP, on each one of the components of $\Omega_i$.

By applying a topogical-combinatorial index lemma (Lemma~\ref{Th:fuctionwithboundary}) and a splitting argumnet, we will conclude that  $\Omega$ may be decomposed into a union (with disjoint interiors) of annuli, quadrilaterals or sliced quadrilaterals. We will finish the proof by showing that the gluing is geometric, i.e. that for all $i$, any part of $\partial \Omega_i$  has the  same  flux-gradient metric (Definition~\ref{de:gradientmetricl1}),  with respect to the boundary value problems induced on each of the two components it belongs to.
 
Throughout this paper we will assume that the reader is familiar with the results and terminology of \cite{Her1}. A theorem, two definitions and a process of modifying a boundary value problem that are essential to the applications of this paper, will be recalled in \S\ref{se:re} in order to make this paper self contained. 
The rest of this paper is organized as follows. 
In \S\ref{se:Qaud}  we prove Theorem~\ref{Th:rectangle}. In \S\ref{se:topo} we prove a topological index lemma which is a generalization of \cite[Theorem 1]{Ba} and  \cite[Theorem 2]{LaVe}. This lemma may be regarded as a discrete version of the Hopf-Poincar\'{e} Index Theorem with additional terms arising from boundary data.   In \S\ref{se:low complexity}, we provide several other low complexity examples and their associated propelled surfaces. Some of these cases may be analyzed without using the topological lemma of \S\ref{se:topo}; however, this lemma provides a uniform approach, hence we will apply it. Finally, \S\ref{se:main} is devoted to the proof of Theorem~\ref{Th:ladder}.  

Since one may view a Dirichlet boundary value problem as a special case of a Dirchlet-Neumann boundary problem, the main result in \cite{Her1} follows from Theorem~\ref{Th:ladder} of this paper.
  However, there is an essential difference between the methods of this paper and those in \cite{Her1}. In some sense the construction of the image surfaces (${\mathcal S}_{\Omega}$) in this paper is considerably less explicit than the construction in \cite{Her1}. At this moment, we are unable to provide a structure theory for level curves of the DN-BVP.  The structure theorem (\cite[Theorem 2.34]{Her1}) for level curves of D-BVP allowed us to cut the domain along the hierarchy  of singular level curves until we obtain simple regions and then glue back.  Thus, in this paper we are led to work with the subdomains $\Omega_i$, which turns out to be sufficient for the applications of this paper. One may view our results and techniques in this paper as well as in \cite{Her1}, as providing purely combinatorial-toplogical analogues to classical counterparts. See in particular \cite[Theorem 4-5]{Ah}, where a connection (in the smooth category) between {\it extremal length} of a family of curves and {\it Dirichlet Energy} of a boundary value problem is exploited.

\begin{Rem}
The assertions of Theorem~\ref{Th:rectangle} may (in principle) be obtained by employing 
 techniques introduced in the famous  paper by Brooks, Smith, Stone and Tutte (\cite{BSST}), in which they study square tilings of rectangles. They define a correspondence between square tilings of rectangles and planar multigraphs endowed with two poles, a source and a sink. They view the multigraph as a network of resistors in which current is flowing. In their correspondence, a vertex corresponds to a connected component of the union of the horizontal edges of the squares in the tiling; one edge appears between two such vertices for each square whose horizontal edges lie in the corresponding connected components. Their approach is based on {\it Kirckhhoff's circuit laws} that are widely used in the field of electrical engineering. We found the sketch of the proof of Theorem~\ref{Th:annulus} given in \cite{BSST} hard to follow. For a summary of other proofs of Theorem~\ref{Th:rectangle}, a bit of the history of this problem, and generalizations, see Remark 0.5 in \cite{Her1} (as well as \cite{BeSch1},\cite{D},\cite{CaFlPa}, and \cite{Ke}).
   We include our proof of Theorem~\ref{Th:rectangle}, which is guided by similar principles to some of the ones mentioned above, yet significantly different in a few  points,  in order to make this paper self-contained. In addition, the important work of Bendito, Carmona and Encinas (see for example \cite{BeCaEn1},\cite{BeCaEn2},\cite{BeCaEn3}) on boundary value problems on graphs
 allows us to use a unified framework to more general problems. Their work is essential to our applications and we have used parts of it quite frequently in \cite{Her1}, and this paper as well as its sequel \cite{Her3}. 
\end{Rem}

\noindent{\small \ack  %{\small \ % We
%are grateful to G\'{e}rard Besson, Francis Bonahon, Gilles Courtois, Dave Gabai, Steven Kerckhoff and Ted Shifrin
  %for enjoyable and helpful discussions on the subject of this paper.} 
  \ Some of the  results of this paper (as well as from \cite{Her1}) were presented at the $G^3$ (New Orleans, January 2007), the  Workshop on Ergodic Theory and Geometry (Manchester Institute for Mathematical Sciences, April 2008), and at the conference ``Around Cannon's Conjecture" (Autrans-France, April 2010). We deeply thank the organizers for the invitations and well-organized conferences.}

\section{A reminder}
\label{se:re}
\paragraph{\bf Finite networks} \label{pa:fn}
In this paragraph we will mostly be using the notation of Section 2 in \cite{BeCaEn}.
Let
$\Gamma=(V,E,c)$ be a planar {\it finite network}, that is a planar, simple, and
finite connected graph with vertex set $V$ and edge set $E$, where each edge $(x,y)\in E$ is
assigned a {\it conductance} $c(x,y)=c(y,x)>0$. Let ${\mathcal P}({V})$
denote the set of non-negative functions on $V$. Given $F\subset V$ we denote by $F^{c}$ its complement in
$V$.  Set
${\mathcal P}(F)=\{u\in {\mathcal P}(V):S(u)\subset F\}$, where $S(u)=\{ x \in V: u(x)\neq 0 \}$.  The set  $\delta F=\{x\in F^{c}: (x,y)\in E\ {\mbox
{\rm for some}}\ y\in F \}$ is called the {\it vertex boundary} of
$F$. Let ${\bar F}=F\bigcup \delta F$ and let $\bar E=\{(x,y)\in
E :x\in F\}$.
Given $F \subset V$, let
${\bar \Gamma}(F)=({\bar F},{\bar E},{\bar c})$ be the network
such that ${\bar c}$ is the restriction of $c$ to ${\bar E}$. 
We say that $x\sim y$ if $(x,y)\in \bar E$.
% For $x\in \bar F$ let $k(x)$ denote the degree of $x$ (if $x\in\delta(F)$ the neighbors of $x$ are taken
  %only  from $F$). 
  
  %For $f,h:\bar E \rightarrow R$ we let 
 
% \begin{equation}
% \label{eq:inner}
% (f,h)=\sum_{(x,y)\in \bar E}\frac{f(x,y)h(x,y)}{c(x,y)}
% \end{equation} be an
%inner product on $l^2(\bar E,1/c)$.

%For $f,h:\bar E \rightarrow R$ we let 
% $(f,h)=\sum_{(x,y)\in \bar E}\frac{f(x,y)h(x,y)}{c(x,y)}$ be an
%inner product on $l^2(\bar E,1/c)$ (see \cite[1.2.A]{Woe}).
\medskip
The following are discrete analogues of
classical notions in continuous potential theory \cite{Fu}.

\begin{Def}\mbox{ \rm (\bf\cite[Section 3]{BeCaEn1})}
 \label{def:energy}  
 Let $u\in {\mathcal P}({\bar  F})$. 
 Then for $x\in \bar F$, the function $\Delta u(x)=\sum_{y\sim x}c(x,y)\left( u(x)-u(y) \right )$ is called
  the Laplacian of $u$ at $x$, \mbox{\rm(}if $x\in\delta(F)$ the neighbors of $x$ are taken
  only  from $F$\mbox{\rm)} and
 the number
 \begin{equation}
 \label{eq:energy}
E(u)= \sum_{x\in\bar F}\Delta u(x)u(x)=\sum_{(x,y)\in \bar E} c(x,y)( u(x)-u(y)
)^2,
\end{equation}
 is called the {\it Dirichlet energy} of $u$.
A function $u\in {\mathcal P}({\bar F})$ is called harmonic in $F\subset V$ if
$\Delta u(x)=0,$ for all $x\in F$.
\end{Def}

A fundamental property which we will often use is the {\it maximum-minimum principle}, asserting that if $u$ is harmonic on $V'\subset V$, where $V$ is a connected subset of vertices having a connected interior, then $u$ attains its maximum and minimum on the boundary of $V'$ (see \cite[Theorem I.35]{So}).

%When $(x,y)\in \bar E$, let $[x,y]$ denote the directed
%edge from $x$ to $y$, and let $\overrightarrow E=\{[x,y] : (x,y)\in \bar E \}$ denote the  set of all directed
%edges. Given $u : V\ra \maths{R}$, we define the {\it differential } or
%the {\it gradient} of $u$ as $\du: \overrightarrow E\ra \RR$ by
%$\du[x,y]=c(x,y)(u(y)-u(x))$ for all $[x,y]\in \overrightarrow E$ (see for instance the notation of  \cite[Section 2]{BeSch2}).
%Using the planarity of $\Gamma$, if
%$|\overrightarrow E|=m$, then $\du$ may be identified with a vector in $\RR^m$.
% It now follows by
%Definition~\ref{def:energy} that for every function
%$u:V\ra \RR$ we have that 
%$I(u)= \frac{1}{2}\sum_{e \in \overrightarrow E}\|\du( e)\|^2=\sum_{e\in\bar E}|\du(e)|^2$.
%Let $\overrightarrow E(x)$ denote the set of all
%edges of the form $[x,y]$ which are in $\overrightarrow E$. As above, any
%$u:\overrightarrow E(x)\ra \RR$ may be naturally viewed as an element in
%$\RR^{k(x)}$. We will denote this vector space, with the restriction of the inner
%product on $\bar E$, by $T_x$. In particular, for $u:\bar F\ra \RR$, we let 
%$ \overrightarrow \du(x)=\du|_{\overrightarrow E(x)}$ denote the restriction of $\du$ to
%$\overrightarrow E(x)$.

For $x\in \delta(F)$, let $\{y_1,y_2,\ldots,y_m\}\in F$ be its neighbors enumerated clockwise.
%The normal vector derivative at $x\in \delta(F)$ is defined by 
%\begin{equation}
%\overrightarrow{\frac{\bord u}{\bord n}}(F)(x)=
%\left( c(x,y_1)(u(x)-u(y_1)),\ldots ,c(x,y_m)(u(x)-u(y_m) ) \right ).
%\end{equation}
 The {\it normal derivative} (see \cite{ChGrYa}) of $u$ at a point
$x\in \delta F$ with respect to a set $F$ is 
\begin{equation}\frac{\bord u}{\bord n}(F)(x)= \sum_{y\sim x,\
y\in F}c(x,y)  (u(x)-u(y)).
\end{equation} 
\smallskip 

%Let
%${\bar\Delta}$ be the Laplacian of ${\bar\Gamma}$. 
The following
 proposition establishes a discrete version of the first classical {\it Green
identity}. It plays a crucial role in the proofs of the main theorems in  \cite{Her, Her1} and is essential to the applications of this paper as well as in its sequel \cite{Her3}.

\begin{Thm} \mbox{\rm \bf(\cite[Prop. 3.1]{BeCaEn}) (The first Green
identity)}
\label{pr:Green id} Let $F \subset V$ and $u,v\in {\mathcal P}({\bar
F})$. Then we have that
\begin{equation}
\label{eq:Green}
 \sum_{(x,y)\in {\bar
E}}c(x,y)(u(x)-u(y))(v(x)-v(y))=\sum_{x\in F}\Delta
u(x)v(x)+\sum_{x\in\delta(F)}\frac{\bord u}{\bord n}(F)(x)v(x).
\end{equation}
\end{Thm}

\paragraph{\bf The flux-gradient metric} \label{pa:fgmetric}
A {\it metric}  on a finite network is a function $\rho : V\ra [0,\infty)$. In particular, the length of a path is given by integrating $\rho$ along the path (see \cite{Ca} and \cite{Du} for a different definition). 
 In \cite[Definition 1.9]{Her1} we defined a ``metric" which will be used throughout this paper. 
\begin{Def}
\label{de:gradientmetricl1}
Let $F\subset V$ and let  $f\in {\mathcal P}({\bar F})$. The {\it flux-gradient metric} is defined by
\begin{equation}
\label{eq:gradient metric}
 \rho(x)=     \frac{\bord f}{\bord n}(F)(x), \ \ \mbox{\rm if}\  x\in
    \delta(F).
    \end{equation}

This definition allows us to define a notion of length to any subset of the vertex boundary of $F$ by declaring: 
\begin{equation}
\label{eq:length gradient metric}
\mbox{\rm Length}(\delta F)=\big |\sum_{x \in \delta F} \frac{\bord f}{\bord n}(F)(x)\big|.
\end{equation}
\end{Def}

%Let $F\subset V$ and let  $f\in {\mathcal P}({\bar F})$. The $l_1$-{\it gradient metric} is defined by
%\begin{equation}
%\label{eq:gradient metric}
% \rho(x)=\left \{ \begin{array}{ll}
%    \|\overrightarrow\df(x)\|_{1}&\mbox{\rm if}\  x\in F \\
%    \\
%    \|\overrightarrow{ \frac {\bord f}{\bord n}}(F)(x)\|_{1}&\mbox{\rm if}\  x\in
 %   \delta(F).
%  \end{array} \right.
%  \end{equation}

In the applications of this paper, we will use the second part of the definition in order to define length of connected components of level curves of a boundary value solution. 
In \cite[Definition 3.3]{Her}, we defined a similar metric ($l_2$-gradient metric) proving several length-energy inequalities.

\paragraph{\bf Simple modifications of a boundary value problem} \label{pa:simple}
We will often need to modify a given cellular decomposition as well as the boundary value problem  associated with it. The need to do this is twofold. 
First, assume for example, that $L$ is a fixed, simple, closed  level curve. 
Since $L\cap {\mathcal T}^{(1)}$ is not (generically) a subset of ${\mathcal T}^{(0)}$, Defintion~\ref{de:gradientmetricl1} may not be directly employed  to provide a notion of length to $L$. We therefore add vertices and edges according the the following procedure. Such new vertices will be called vertices of type I.

Let ${\mathcal O}_1, {\mathcal O}_2$ be the two distinct connected components of $L$ in $\Omega$ with $L$ being the boundary of both (these properties follow by employing the Jordan curve theorem). We will call one of them, say ${\mathcal O}_1$, an {\it interior domain} if all the vertices which belong to it have $g$-values that are smaller than the $g$-value of $L$. The other domain will be called the {\it exterior domain}. Note that  by the maximum principle, one of ${\mathcal O}_1, {\mathcal O}_2$ must have all of its vertices with $g$-values smaller than $L$. 

Let $e\in {\mathcal T}^{(1)}$ and $x=e\cap L$. For $x\not\in {\mathcal T}^{(0)}$, we now have two new edges $(x,v)$ and $(u,x)$. We may assume that $v\in {\mathcal O}_1$ and $u\in {\mathcal O}_2$. We now define conductance constants $\tilde c(v,x)$ and $\tilde c(x,u)$ by 

\begin{equation}
\label{eq:prehar}
\tilde c(v,x)= \frac {c(v,u) (g(v)-g(u))}{g(v)-g(x)} \ \   \mbox{\rm and}\ \  \tilde c(u,x) = \frac {c(v,u) (g(u)-g(v))}{g(u)-g(x)}.
\end{equation}

 By adding to ${\mathcal T}$ all the new vertices and edges, as well as the piecewise arcs of $L$ determined by the new vertices, we obtain two cellular decompositions, ${\mathcal T}_{{\mathcal O}_1}$ of  
${\mathcal O}_1$ and  ${\mathcal T}_{{\mathcal O}_2}$ of ${\mathcal O}_2$. Also, two conductance functions are now defined on the one-skeleton of these cellular decompositions by modifying the conductance function for $g$ according to Equation~(\ref{eq:prehar}) (i.e. changes are occurring only on new edges). One then follows the arguments preceding \cite[Definition 2.7]{Her1} and defines a modification of the given boundary value problem, the solution of which is easy to control (using the existence and uniqueness Theorems in \cite{BeCaEn}). Second, it is easy to see that Theorem~\ref{pr:Green id} may not be directly applied for a modified  cellular decomposition and the modified boundary value problem defined on it. Informally, the modified graph of the network needs to have its boundary components separated enough, in terms of the combinatorial distance, in order for Theorem~\ref{pr:Green id} to be applied.  In order to circumscribe such cases, we will add enough new vertices along edges and change the conductance constants along new edges in the obvious way, i.e. the original solution will still be harmonic at each new vertex and will keep its values at the two vertices along the original edge.  Such new vertices will called type II.

\section{The case of a quadrilateral}
\label{se:Qaud} 

 We now describe the structure of the proof of Theorem~\ref{Th:rectangle}.  The proof consists of two parts. First, we will show that there is a well-defined mapping from ${\mathcal T}^{(1)}$ into a set of 
 (Euclidean) rectangles embedded in the recatngle $S_{\mathcal R}$. The crux of this part is the fact  that level curves of $g$ have the same induced length (measured with the  flux-gradient metric),  and  a simple application of the maximum principle. Second, we will show that the collection of these rectangles forms a tiling of $S_{\mathcal R}$ with no gaps. The dimensions of $S_{\mathcal R}$ and the first Green identity (Theorem~\ref{pr:Green id}) will allow us to end the proof by employing an energy-area computation.

Keeping the notation of the introduction, we let $A,B,C,D$ be the corners of ${\mathcal R}$ ordered clockwise with $A$ being the left lower corner and with  $AB={\mathcal R}_{\mbox{\tiny \rm   left}}, BC={\mathcal R}_{\mbox{\tiny \rm top}}, CD={\mathcal R}_{\mbox{\tiny \rm right}}$ and  $DA={\mathcal R}_{\mbox{\tiny \rm bottom}}$ being the boundary arcs decomposition of $\partial{\mathcal R}$.

\begin{Pro}
\label{pr:topoflevel}
For each $s\in [0,k]$, the associated $g$-level curve, $l_s$, is simple and parallel to $AB$, i.e. its endpoints lie on $BC$ and $DA$, respectively, and it does not intersect $AB\cup CD$.
\end{Pro}
\begin{proof}
Harmonicity of $g$ implies that there exists a path in ${\mathcal T}^{(1)}$ from $B$ to $CD$  along which $g$ is strictly increasing. Since $g$ is extended linearly over ${\mathcal T}^{(1)}$, each value in $[0,k]$ is attained (perhaps at some point on an edge of this path). Let $s_0\in [0,k]$ be any such value. The assertion of the proposition is certainly true for $s_0=0$ and for $s_0=k$. Therefore assume that $s_0\in (0,k)$ and that it is attained 
at some point which we will denote by $v_0$. By construction $v_0$ is not an endpoint for $l_{s_0}$ unless $v_0\in \partial{\mathcal R}$, and it is clear that $v_0\not\in AB\cup CD$. Extend $l_{s_0}$ from $v_0$ through triangles and quadrilaterals to a line. It follows by the maximum principle that $l_{s_0}$ is simple and it is not a circle. Also, the intersection of $l_{s_0}$ with each $2$-cell is a line segment whose intersection with the boundary of this cell consists of exactly two points, or a vertex. Since ${\mathcal T}^{(2)}$ is finite, $l_{s_0}$ is a closed, connected interval, and by construction may have its endpoints only  in $\partial{\mathcal R}$. Let $P_{v_0}$ and $Q_{v_0}$ be its endpoints.
To finish the proof, we need to show that $P_{v_0}$ and $Q_{v_0}$ do not belong to the same boundary arc of $\partial{\mathcal R}$. It is clear that none of the endpoints can belong to $AB\cup CD$, so suppose (without loss of generality) that they belong to $BC$. Let $l=l(P_{v_0}, v_1, v_2,\dots, Q_{v_0})$ be the path in $BC$ connecting $P_{v_0}$ to $Q_{v_0}$, and let ${\mathcal P}_{s_0}$ be the polygon formed by $l_{s_0}$ and the arc $l$. Attach a copy of it, $\bar {\mathcal P}_{s_0}$, along $l$. The result is a polygonal disc ${\mathcal D}_{s_0}$ all of its boundary vertices having the same $g$-value, $s_0$.  

Let $\bar g$ be the function which is defined on ${\mathcal D}_{s_0}$ by letting $\bar g=g$ on ${\mathcal D}_{s_0}\cap ({\mathcal R}\cup \partial{\mathcal R})$ and by letting $\bar g(\bar v)=g(v)$ for every $\bar v$ in the attached copy where $v\in {\mathcal P}_{s_0}$ is the combinatorial symmetric ``reflection" of $\bar v$. 
By changing the conductance constants (only) along edges in $l$, the fact that $g$ is harmonic in ${\mathcal R}$ and since   
\begin{equation}
\label{eq;combiref}
\frac{\partial g}{\partial n}({\mathcal P}_{s_0})(v)=0, 
\end{equation} 
for every $v\in l$. It easily follows that $\bar g$ is harmonic in ${\mathcal D}_{s_0}$. However,  $\bar g$ has constant boundary values, hence $\bar g$ must be a constant (by the maximum-minimum principle). This is absurd.

\end{proof}

\begin{figure}[htbp]
\begin{center}
 \scalebox{.45}{ \input{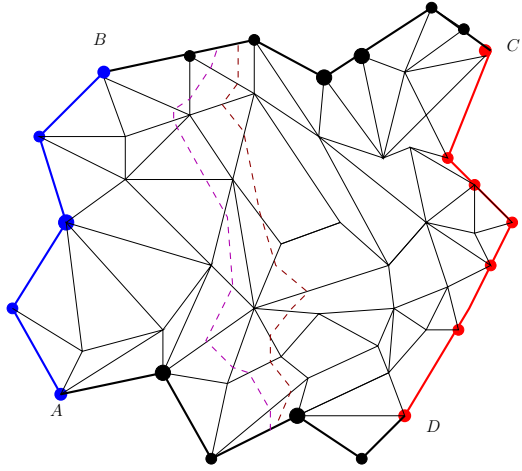}}
 \caption{A quadrilateral and two parallel level curves.}
\label{figure:quad}
\end{center}
\end{figure}

\begin{Rem}
Similarly, it follows by the harmonicity of $g$ that for each $s\in [0,k]$, $l_s$ is unique. 
\end{Rem}

One useful consequence of Proposition~\ref{pr:topoflevel} is that for each $s\in (0,k)$ the level curve $l_s$ separates ${\mathcal R}$ into two quadrilaterals having disjoint interiors. The first has its left boundary equal  to $AB$, its right boundary equal to $l_s$, its top boundary being the part of the top boundary of ${\mathcal R}$ connecting $B$ to the endpoint of $l_s$ on it, and its bottom boundary being part of the bottom boundary of ${\mathcal R}$ connecting $A$ to the second endpoint of $l_s$. We will denote this quadrilateral by ${\mathcal O}_{1}$ and its complement in ${\mathcal R}$ by ${\mathcal O}_{2}$. 

The length of a curve with respect to the flux-gradient metric 
(see Definition~\ref{eq:length gradient metric}),  which lies on the boundary of two regions, may be computed according to each one of them. The length will be said to be well-defined if it does not depend on the region chosen for carrying the computation.

\begin{Pro}
\label{pr:lengthoflevel}
For every $s\in [0,k]$, the length of its associated level curve $l_s$ with respect to the flux-gradient metric, is well defined and is equal to
\begin{equation}
\label{eq:len}
H=\sum_{v\in CD}\frac{\partial g}{\partial n}({\mathcal R})(v).  
\end{equation} 
Furthermore, the following equality holds
\begin{equation}
\label{eq:lenfromright}
\sum_{v\in CD}\frac{\partial g}{\partial n}({\mathcal O}_{2})(v)=-\sum_{v\in AB}\frac{\partial g}{\partial n}({\mathcal O}_{1})(v) .  
\end{equation} 
\end{Pro}
\begin{proof}
For $s=k$ the first assertion follows from the definition of the flux-gradient metric induced by $g$. Let $s$ be any other value in $[0,k)$, and let $l_s$ be its associated level curve. Let 
\begin{equation}
\label{eq:newvertices}
{\mathcal V}_{s} = l_s\cap{\mathcal T}^{(1)},
\end{equation}
and define $\bar g$ at each point of the set ${\mathcal V}_{s}$ so that $\bar g(v)=g(v)$, for every $v\in {\mathcal T}^{(0)}$, and conductance constants on the added edges so that $\bar g$ is harmonic at each $v\in {\mathcal V}_{s}$, and 
\begin{equation}
\label{eq:neumannderi}
\frac{ \partial \bar g}{\partial n}(\xi_1)=\frac{ \partial \bar g}{\partial n}(\xi_2)=0,
\end{equation}
where $\xi_1$ and $\xi_2$ are the endpoints of $l_s$ on BC and AD, respectively (the very last modifications are needed only if $\xi_1$ and $\xi_2$ are not in ${\mathcal T}^{(0)})$. Note that the set ${\mathcal V}_{s}$ comprises (generically)  vertices of type I (see the last paragraph in Section~\ref{se:re}).

We now apply Green's Theorem (Theorem~\ref{pr:Green id}) with $u=\bar g$ and $v\equiv 1$ over the quadrilateral ${\mathcal O}_{2}$, which is determined by $l_s, \xi_1C, CD$ and $D\xi_2$ to obtain
\begin{eqnarray}
\label{greenapplied}
&0=\sum_{x\in {\mathcal O}_{2}}\frac{\partial\bar g}{\partial n}({\mathcal O}_{2})(x)= \sum_{t\in l_s}\frac{\partial\bar g}{\partial n}({\mathcal O}_{2})(t)+ \sum_{p\in CD}\frac{\partial\bar g}{\partial n}({\mathcal O}_{2})(p)=& \\  & &\nonumber \\&  \sum_{t\in l_s}\frac{\partial\bar g}{\partial n}({\mathcal O}_{2})(t)+\sum_{p\in CD}\frac{\partial g}{\partial n}({\mathcal O}_{2})(p),&\nonumber
\end{eqnarray}
and hence that 
\begin{equation}
\label{eq:l1is equal}
|\sum_{t\in l_s}\frac{\partial\bar g}{\partial n}({\mathcal O}_{2})(t)|= |\sum_{p\in CD}\frac{\partial g}{\partial n}({\mathcal O}_{2})(p)|=\sum_{p\in CD}\frac{\partial g}{\partial n}({\mathcal O}_{1})(p)= \sum_{p\in CD}\frac{\partial g}{\partial n}({\mathcal R})(p),
\end{equation}
which completes the proof of the first assertion. 
Note that in order to apply Green's Theorem we may need to add vertices of type II and change the conductance constants along added edges (see the last paragraph in Section~\ref{se:re}).

By applying Green's Theorem (Theorem~\ref{pr:Green id}) to ${\mathcal R}$ one obtains the second assertion (as is the case with the last equality in Equation~(\ref{eq:l1is equal}), both sides of the equation have the same value when computed relative to ${\mathcal R}$). In particular, this means that the computation of the length of $l_s$ with respect to the flux-gradient metric does not depend on which one of the two quadrilaterals, 
${\mathcal O}_1$ or ${\mathcal O}_2$,  it is carried.

\end{proof}

Given a Euclidean rectangle $Q=[0,W]\times [0,H]$ embedded in the Euclidean plane, we will endow it with the naturally induced coordinates. Its boundary components $[0,W]\times\{0\}, \{0\}\times[0,H], [0,W]\times \{H\}$ and $\{W\}\times[0,H]$ will be called {\em bottom, left, top} and {\em right}, respectively. 
Before providing  the proof of Theorem~\ref{Th:rectangle}, we need a definition which will simplify keeping track of the mapping $f$. 
\begin{Def}
\label{de:marker}
A marker on a Euclidean rectangle is a horizontal closed interval  which is the isometric image of  $[a,b]\times \{t\}$, for some $t \in [0,H]$ and $[a,b]\subset [0,W]$ with $a<b$. The marker's  leftmost end-point corresponds to $(a,t)$ and its rightmost end-point to $(b,t)$.
\end{Def}

\noindent{\bf Proof of Theorem~\ref{Th:rectangle}.}
Let $S_{\mathcal R}$ be a straight Euclidean rectangle with 
width $W=k$ and height 
\begin{equation}
\label{eq:height1}
H=\sum_{x\in\partial{\mathcal R}_{\mbox{\tiny\rm right}}}\frac{\bord g}{\bord n}(\mathcal R)(x).
\end{equation}  
 %(The right hand side  should be thought of as the discrete {\it flux} of $g$ through $E_1$).
Let ${\mathcal L}=\{L_1,\ldots,L_k\}$ be the level sets of $g$ corresponding to the vertices in ${\mathcal T}^{(0)}$ arranged in descending $g$-values order.  We  add a vertex at each intersection of an edge with an $L_i,\   i=1,\ldots, k,$ (which is not already a vertex in ${\mathcal T}^{(0)}$), and if necessary more vertices on edges so that any two successive level curves in ${\mathcal L}$ are at combinatorial distance (at least) two. As before, the first group of added vertices is of type I and the second is of type II.
%(recall that conductance along edges are changed as well, according to the discussion preceding Definition 2.7 in \cite{Her1}).
 %In particular, the assertions of Theorem~\ref{th:lengthsareequal} and 
%Remark~\ref{re:lengthof} hold. Since $\chi({\mathcal A})=0$ and the index of a singular vertex is always negative, the index formula (Theorem ~\ref{Th:index}) prevents the existence of singular vertices in ${\mathcal A}$.  Therefore, ${\mathcal K}=\{0,k\}$,  and furthermore the length of any $g$-level curve is equal to $C$.

Starting with $x_1=C$, we order the vertices $\{x_1=C,\ldots,x_p=D\}$ in $L_1 (=CD)$, as well as the vertices on any other level curve, in a monotone decreasing order.
Let $\{y_1,y_2,\ldots, y_t\}$ be the type I neighbors of $x_1$ in the new cellular decomposition oriented counterclockwise (which will henceforth be assumed to be the ordering of the neighbors of any vertex). 
We identify $x_1$ with $(k,H)$ in the coordinates mentioned above, and associate markers   
$\{m_{x_1,y_1},\ldots,m_{x_1,y_t}\}$ with $x_1$ in the following way. For $s=1,\ldots,t$, the length of the marker $m_{x_1,y_s}$  is equal to (the constant) $g(x_1)-g(y_s)$ and its rightmost end-point is positioned on the right boundary of $S_{\mathcal A}$ at height

\begin{equation}
\label{eq:positionmarker}
H - 
\sum_{k=1}^{s-1} c(x_1,y_k)(g(x_1)-g(y_k)).
\end{equation} 
  For each edge $e_{u,v}=[u,v]$ with $g(u)>g(v)$, let $Q_{u,v}$ be a Euclidean rectangle with width equal to $g(u)-g(v)$ and height equal to $c(u,v)(g(u)-g(v))$. We will identify a  Euclidean rectangle and  its image under an isometry. 
For $s=1,\ldots,t$,
 we position $Q_{x_1,y_s}$ in $S_{\mathcal A}$ in such a way that its top boundary edge coincides with $m_{x_1,y_s}$. By construction and the position of the markers,
 
\begin{equation}
\label{eq;intersectionis1}
Q_{x_1,y_s}\cap Q_{x_1,y_{s+1}}= m_{x_1,y_{s+1}}.
\end{equation}
Assume that we have placed markers and  rectangles associated to all the vertices up to $x_k$  where $k<p$; let $z_1$ be the uppermost neighbor of $x_{k+1}$ and let $Q_{x_{k},v}$ be the lowermost rectangle associated with $x_k$. That is, $v$ is the lowermost vertex which is a neighbor of $x_k$ (it may of course happen that $v=z_1$). We now position the marker $m_{x_{k+1},z_1}$ 
so that it is lined with the bottom boundary edge of $Q_{x_{k},v}$, and its rightmost end-point is on the right boundary of $S_{\mathcal A}$ at height which is given by the obvious modification of Equation~(\ref{eq:positionmarker}).  
We continue  placing markers and rectangles corresponding to the rest of the neighbors of $x_{k+1}$, and terminate these steps when $k=p$. Note that the right boundary of $S_{\mathcal A}$ is completely covered by the right boundary edges of the rectangles constructed above, where intersections between any two of these  edges is either a vertex or empty.

\begin{figure}[htbp]
\begin{center}
 \scalebox{.55}{ \input{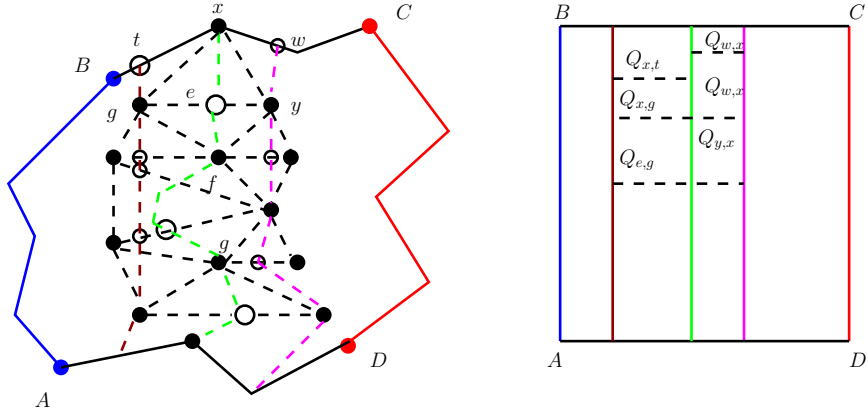}}
 \caption{Several rectangles in $S_{\mathcal R}$  after the completion of the construction.}
\label{figure:markers}
\end{center}
\end{figure}

For all $1<n<k$, assume that all the markers corresponding to vertices in $L_{n-1}$ and their associated rectangles have been placed as above in such a way that the following conditions, which we call {\it consistent},  
  hold. For $[w,v]\in T^{(1)}$ with $g(w)>g(v)$ and $s\in [w,v]$ a vertex of type I, the rightmost end-point of the marker $m_{s,v}$ coincides with the leftmost end-point of the marker $m_{w,s}$; moreover, the union of the rectangles $Q_{w,s}, Q_{s,v}$ tile $Q_{v,w}$.
   Informally, if these conditions are met, this will allow us to ``continuously extend" rectangles associated with edges that cross level curves along these curves, and therefore will show that edges in ${\mathcal T}^{(1)}$ are mapped in one to one fashion (perhaps in several steps) onto a unique rectangle,  $Q_{w,v}$.

We will now describe how to place the markers and rectangles corresponding to the vertices of the level set $L_{n}$, for $n>2$. The rightmost end-point of each marker associated with a vertex $v\in L_n$ and any of its neighbors in $L_{n+1}$ is placed in $S_{\mathcal R}$ on the vertical line corresponding to $g(v)$ (the actual height on this level curve is computed by a formula which is an easy modification of Equation~\ref{eq:positionmarker}). Observe that $v$ is a vertex in some  $[q_i,v]$, where $q_i$ belongs to $L_{n-1}$. Choose among all such edges the uppermost (viewed from $v\in L_n$). Let $[q_0,v]$ be this edge and let $m_{q_{0},v}$ be its marker. Place $m_{v,w}$,  the marker of $v$ which corresponds to an edge $[v,w]$, with $w$ being the uppermost vertex among the neighbors of $v$ in $L_{n+1}$, so that its rightmost end-point coincides with the leftmost end-point  of $m_{q_0,v}$. To conclude the construction, continue as above, exhausting all the  markers emanating from $v$, and vertices in $L_{n}$.

By the maximum principle, our construction, and the fact that all level curves have their lengths (with respect to the flux-gradient metric) equal to $H$, it is clear that the union of the rectangles is contained in $S_{\mathcal R}$.

\begin{Pro}
\label{pr:consistent}
The placement of rectangles associated to the construction of markers as described above is consistent.
\end{Pro}
\dem 
We prove the assertion by induction on the set of  level curves. The assertion is obviously true for all rectangles associated with markers emanating from $L_1$, since no such marker is a continuation of another one.  Let $s_0$ be a vertex of type I on $L_n$, $n>2$. By definition, $s_0$ is connected to a unique vertex $v_0\in L_{n+1}$ and to a unique vertex $w_{0}\in L_{n-1}$. 
We first consider the case in which $s_0$ is the only type I vertex on $L_n$. It is easy to check that the following system of equations with the two unknowns ${\tilde c}(w_0,s_0)$ and ${\tilde c}(s_0,v_0)$ has a non-trivial solution
\begin{eqnarray}
\label{eq:consistenta}
{\tilde c}(w_0,s_0)(g(w_0)-g(s_0)) & = &  {\tilde c}(s_0,v_0)(g(s_0)-g(v_0)),\ \mbox{\rm and}\\
{\tilde c}(w_0,s_0)(g(w_0)-g(s_0)) & = &c(w_0,v_0)(g(w_0)-g(v_0))\nonumber.
\end{eqnarray}

The unknowns present the conductance constants to be assigned to $[w_0,s_0]$ and $[s_0,v_0]$, respectively, so that the modified DN-BVP solution function $\bar g$ is harmonic at $s_0$. By the construction of the rectangles, this implies that the height of $Q_{s_0,v_0}$ is the same as the height of
$Q(w_0,s_0)$, so that they can be glued along the appropriate edges. The second equation reflects that the height 
of the rectangles associated to $m_{w_0,s_0}$ and $m_{s_0,v_0}$ is equal to the height of the rectangle which one would associate to $m_{w_0,v_0}$ (in the case that there were no type I vertices on $[w_0,v_0]$).  In other words, the construction of rectangles is consistent and once an edge is split by a type I vertex, the two constructed rectangles my be glued along the appropriate edges; thus we obtain the same effect as constructing the rectangle associated to the original edge.  Note that since $g(w_0)-g(v_0)=(g(w_0)-g(s_0))+(g(s_0)-g(v_0))$, matching the widths of the above rectangles is not an issue.

Assume now that $s_q$ is the first vertex of type I in $L_n$ which is lower than $s_0$. By definition, $s_q$ is connected to a unique vertex $w_p$ in $L_{n-1}$ and to a unique vertex $v_l \in L_{n+1}$. Let $\{s_1,\ldots,s_{q-1}\}$ be the vertices in $L_n$ between $s_0$ and $s_q$, and let $\{w_1,\ldots, w_{p-1}\}$ be the vertices in $L_{n-1}$ between $w_0$ and $w_p$. 
Let ${\mathcal Q_1}={\mathcal Q}_{w_0,s_0,s_q,w_p}$ be   
the quadrilateral  enclosed by $[w_0,s_0]\cup [w_p,s_q]\cup L_{n-1}\cup L_{n}$, and which contains 
$\{w_0,\ldots,w_p\}$, and let ${\mathcal Q}_2={\mathcal Q}_{s_0,v_0,s_q,v_l}$ be  
the quadrilateral enclosed by $[s_0,v_0]\cup [s_q,v_l]\cup L_{n}\cup L_{n+1}$, and which contains $\{s_0,\ldots,s_q\}$.

\begin{figure}[htbp]
\begin{center}
 \scalebox{.45}{ \input{consistent.pstex_t}}
 \caption{Viewing $Q_1\cup Q_2$.}
\label{figure:consi}
\end{center}
\end{figure}

In order to prove that the consistent conditions hold for all markers and rectangles created in this step, it suffices to prove it at $s_q$; assuming (without loss of generality) that the first marker associated with vertices in $L_n$, that was placed in a consistent way, is $m_{s_0,v_0}$. 
By the construction of the markers (see in particular 
Equation~(\ref{eq:positionmarker}) suitably adapted) we need to prove that

\begin{eqnarray}
\label{eq:consistent}
& \sum_{i=1}^{p-1}\frac{\partial \bar g}{\partial n}({\mathcal Q}_1)(w_i)+\frac{\partial\bar  g}{\partial n}({\mathcal Q}_1)_{\mbox {\rm {\tiny low-left}}}(w_0) +\frac{\partial \bar g}{\partial n}({\mathcal Q}_1)_{\mbox {\rm {\tiny top-left}}}(w_p)& = \\
&\sum_{i=1}^{q-1}\frac{\partial \bar g}{\partial n}({\mathcal Q}_2)(s_i)+\frac{\partial \bar g}{\partial n}({\mathcal Q}_2)_{\mbox {\tiny\rm low-left}}(s_0)+\frac{\partial \bar g}{\partial n}({\mathcal Q}_2)_{\mbox {\tiny\rm top-left}}(s_q),&\nonumber
\end{eqnarray}
where the subscripts ``low-left"  and ``top-left" are posted to emphasize that neighbors in the expressions  are taken from  ${\mathcal Q_1}$  or ${\mathcal Q_2}$ only. It is easy to check that since  $\bar g$ is harmonic at each $s_i, i=0,\ldots q$ (as well as elsewhere), and since $s_0$ and $s_q$ are type I vertices, Equation~(\ref{eq:consistent}) holds.  

\eop{Proposition~\ref{pr:consistent}}

\medskip

We will now finish the proof by showing that the collection of rectangles constructed above tiles $S_{\mathcal R}$ leaves no gaps. Without loss of generality, suppose that the collection of the rectangles described above  does not cover a strip of the form
$[t_1,t_2]\times [h_0,h_1]$ in $S_{\mathcal R}$, where $0\leq t_1<t_2\leq k$ and $0\leq h_0<h_1\leq H$.  Since $g$ is harmonic there exists at least one path whose vertices belong to ${\mathcal T}^{(0)}$ such that the values of $g$ along this path are strictly decreasing from $k$ to $0$. In particular, the value $t_2$ is attained at some interior point of an edge or at some vertex of this path. By construction, a gap in a $g$-level curve (i.e. an arc of a level curve which is not covered by the left edges of rectangles) never occurs when $t_2$ is the $g$-value associated to a vertex in the modified cellular decomposition, $\tilde{\mathcal T}$. 

Hence, we may assume that $t_2$ is attained in the interior of an edge. Let $L_{t_2}$ be the corresponding level curve. Recall that $L_{t_2}$ is simple and parallel to $CD$ with its two endpoints belonging to $AD$ and $BC$, respectively (as is the case with all other level curves of $g$; see Proposition~\ref{pr:topoflevel}).
We now follow the construction at the beginning of the proof, and let 
$\{u_1,u_2,\ldots, u_q\}$ be all the new vertices on $L_{t_2}$, that is we place a vertex at each intersection of an edge in $\tilde{\mathcal T}^{(1)}$ with $L_{t_2}$.  By Proposition~\ref{eq:len} the length of $L_{t_2}$ (with respect to the flux-gradient metric) is equal to $H$. Moreover, this length is
equal to
\begin{equation}
\label{eq:lengthofgap}
\sum_{i=1}^{q} \frac{\partial g}{\partial n}({\mathcal  O_1})(u_i),
\end{equation}
where ${\mathcal O}_1$ is the interior of the rectangle enclosed by $AB$, (part of) $BC$, $L_{t_2}$ and (part of) $DA$ (see the end of the proof of Proposition~\ref{pr:lengthoflevel}). 
In particular, in principle we may now place in a consistent way, markers and rectangles associated to the collection of edges emanating from the vertices $\{u_1,u_2,\ldots, u_q\}$ so that $L_{t_2}$ is completely covered by the left edges of these rectangles. Since $g$ is extended affinely over edges, every value between $h_0$ and $h_1$ is attained by $g$. Repeating this argument shows that all level curves are covered by rectangles. Hence the collection of rectangles leaves no gaps in $S_{\mathcal A}$.

Using an area argument, we now finish the proof by showing that there is no overlap between any two of the rectangles. Let ${\mathcal U}$ be the union of all the constructed rectangles. By definition,
\begin{equation}
\label{eq;arecomputation}
\mbox{\rm Area}{(\mathcal U)}=\sum_{[x,y]\in \tilde{\mathcal T}^{(1)}}  c(x,y)(g(x)-g(y))(g(x)-g(y)).
\end{equation}
Note that the sum appearing in the right-hand side of Equation~(\ref{eq;arecomputation}) is computed over $\tilde{\mathcal T}^{(1)}$, the induced cellular decomposition. A simple computation (using Equation~(\ref{eq:consistent})) and the fact that the construction is consistent shows that this sum is equal to the one taken over $[x,y]\in {\mathcal T}^{(1)}$. Hence, the right-hand side of this equation is the energy $E(g)$ of $g$ (see Definition~\ref{def:energy}). Therefore, by the first Green identity, applied with $u=v=g$ (see Theorem~\ref{pr:Green id}), the boundary conditions imposed on $g$, and the dimensions of $S_{\mathcal R}$,  we have 
\begin{equation}
\label{eq:areasareequal}
E(g)=\mbox{\rm Area}{(\mathcal U)}=\mbox{\rm Area}(S_{\mathcal R}).
\end{equation}
Hence, since the union of the rectangles do not leave gaps, and are all contained in $S_{\mathcal R}$, they must tile $S_{\mathcal R}$. It is also evident that the mapping $f$ constructed is energy preserving. 
%Recall that up to this moment, markers and rectangles were associated only to vertices of type I in addition to those in ${\mathcal T}^{(0)}$. It is very easy to check that this can be done in a consistent way (using Equation~(\ref{eq:prehar})). 

\eop{Theorem~\ref{Th:rectangle}}

\medskip

\begin{Rem}
\label{re:singular}
In the forthcoming applications and examples of this paper, we will often need to use a slight generalization of Theorem~\ref{Th:rectangle}. First, we will need to allow the 
domain to be  sliced. That is, a quadrilateral, two of his adjacent vertices that belong to the left boundary (or right boundary), are identified, and possibly a finite number of points on the right boundary (or left boundary) are also identified (not necessarily to the same point). See the next section and in particular Example~\ref{ex:pairwithouter}.

One considers a sliced quadrilateral as a quotient of a quadrilateral in the obvious way. The construction of Theorem~\ref{Th:rectangle} goes through with the image being a Euclidean rectangle under the appropriate quotient. Note that some of the image rectangles are not going to be embedded. However, the embedding fails in a controlled way. For a similar situation in the case of a planar pair of pants, and corresponding analysis for higher genus cases in the setting of D-BVP, see \cite[Section 4 ]{Her}.
\end{Rem}

\section{An index lemma} 
\label{se:topo} 
 
Let ${\mathcal G}$ be a polyhedral surface with  (possible empty)   boundary $\partial{\mathcal G}$ . Let $f: {\mathcal G}^{(0)}\rightarrow \RR\cup\{0\}$ be a function such that any two adjacent vertices are given different values. 
Let $v\in {\mathcal G}^{(0)}$ with $v\not\in\partial{\mathcal G}$, and let $w_1,w_2,\ldots,w_k$ be its $k$ neighbors enumerated counterclockwise. Following \cite[Section 3]{LaVe}, consider the number of sign changes in the sequence $\{ f(w_1)-f(v),f(w_2)-f(v),\ldots,f(w_k)-f(v),f(w_1)-f(v)\}$, which we will denote by  $\mbox{\rm Sgc}_{f}(v)$. The index of $v\in {\mathcal G}$ is defined as 
\begin{equation}
\label{eq:index}
\mbox{\rm Ind}_{f}(v)= 1- \frac{\mbox{\rm Sgc}_{f}(v)}{2}.
\end{equation}
For the applications of this paper we need to consider the situation in which $\partial{\mathcal G}\neq \emptyset$.
Let $\bar v\in \partial {\mathcal G}$ and let
$q_1,q_2,\ldots,q_l$ be its neighbors in ${\mathcal G}$ enumerated counterclockwise. Consider the number of sign changes in the sequence $\{ f(q_1)-f(\bar v),f(q_2)-f(\bar v),\ldots,f(q_l)-f(\bar v)\}$, which we will keep denoting by  $\mbox{\rm Sgc}_{f}(\bar v)$. The index of $\bar v\in {\partial \mathcal G}$ is defined as 
\begin{equation}
\label{eq:indexboundary}
\mbox{\rm Ind}_{f}(\bar v)= \dfrac{1}{2}\left(1- \frac{2\  \mbox{\rm Sgc}_{f}(\bar v)}{2}\right).
\end{equation}

\begin{Def}
A vertex whose index is different from zero will be called  singular; otherwise the vertex is regular. A level set which contains at least one singular vertex will be called singular; otherwise the level set will be called regular.
\end{Def} 

A nice connection between the combinatorics and the topology is provided by the following theorem, which may be considered as a discrete Hopf-Poincar\'{e} Theorem. 

\begin{Thm} \mbox{\rm (\cite[Theorem 1]{Ba}, \cite[Theorem 2]{LaVe}) ({\bf An index formula)}}
\label{Th:index} Suppose that ${\mathcal G}$ is closed, then we have
\begin{equation}
\label{eq:Euler}
\sum_{v\in {\mathcal G}}\mbox{\rm Ind}_{f}(v)=\chi({\mathcal G}).
\end{equation}
\end{Thm}

\begin{Rem}
\label{re:withboundary}
Note that due to the topological invariance of $\chi({\mathcal G})$ once the equation above is proved for a triangulated polyhedron, it holds (keeping the same definitions for $\mbox{\rm Sgc}_{f}(\cdot)$ and $ \mbox{\rm Ind}_{f}(\cdot)$ as well as the assumption on $f$) for any cellular decomposition of $\chi({\mathcal G})$. Also, while the theorem above is stated and proved for a closed polyhedral surface, it is easy to show that it holds in the case of a surface with boundary, where there are no singular vertices on the boundary (simply by doubling along the boundary).
\end{Rem}
We now prove a generalization of Theorem~\ref{Th:index} which includes the case of singular vertices on the boundary as well as the case in which $f$ admits constant values  on some arcs of the boundary. Some immediate applications of our generalization will be provided in \S\ref{se:low complexity} and \S\ref{se:main} providing the control we need on the number of critical points as well as their indices.

\begin{Lem}
\label{Th:fuctionwithboundary}
Let $\Omega$ be a bounded, planar, $n$-connected domain with $\partial \Omega$ as its boundary. Suppose that $\Omega\cup\partial\Omega$ is endowed with a cellular decomposition, denoted by ${\mathcal T}$, in which each $2$-cell is a triangle or a quadrilateral.  Suppose that $l$ closed and disjoint arcs are specified on the outer boundary of $\partial\Omega$, and that $m$ closed and disjoint arcs are specified on the other boundary components of $\partial \Omega$.

Let $f:{\mathcal T}^{(0)}\rightarrow \RR^{+}$ be a function which satisfies the following:
\begin{enumerate}
\item $\max(f)$ is attained exactly at each vertex in ${\mathcal T}^{(0)}$ which lies on any of the $l$ arcs,
\item $\min(f)$ is attained exactly at each vertex in ${\mathcal T}^{(0)}$ which lies on any of the $m$ arcs, and
\item any two adjacent vertices in ${\mathcal T}^{(0)}$, other than the ones in  \mbox{\rm (1)} and \mbox{\rm (2)}  have different $f$-values.
\end{enumerate}
Then we have
\begin{equation}
\label{eq:indexwithboundary}
\sum_{v\in \Omega}\mbox{\rm Ind}_{f}(v)+\sum_{\bar v\in \partial\Omega}\mbox{\rm Ind}_{f}(\bar v)+\dfrac{l+m}{2}=\chi(\Omega).
\end{equation}
\end{Lem}

\begin{proof}
We first collapse each one of the arcs in $\partial\Omega$ on which $f$ attains a maximum or a minimum value to a single vertex. The resulting planar domain $\Omega '$ is bounded and  $n$-connected. 
The cellular decomposition ${\mathcal T}$ is changed to a new one ${\mathcal T}'$ in the following way. Any triangle in ${\mathcal T}^{(2)}$ with two of its vertices having the same $f$ value is turned into a digon. Every quadrilateral with two of its vertices having the same $f$ values is turned into a triangle. These changes occur (if at all) only  at combinatorial distance which is equal to one from the arcs on which $f$ attains a constant value. We now collapse all digons and multi-gons connecting two vertices (one of which is on $\partial \Omega '$) to a single edge connecting these vertices. In particular, we have that  $\chi(\Omega)=\chi(\Omega ')$, and ${\mathcal T}'$ is comprised of triangles and quadrilaterals. Furthermore, $f$ attains its maximum on exactly $l$ vertices in the outer boundary of $\partial \Omega '$ and its minimum on exactly $m$ vertices in the inner boundary of $\partial \Omega '$. Also, any two adjacent vertices in  ${\mathcal T'}^{(0)}$ have different $f$-values. 
The indices and the number of the singular interior vertices as well as the indices and the number of singular vertices that are on $\partial \Omega '$ and are not (global) maximum or minimum vertices has not changed. 

We now double $\Omega '$ along its boundary $\partial \Omega '$ and obtain a closed polyhedral surface ${\mathcal G}$ of genus $\chi({\mathcal G})=2\chi({\Omega '})$. The index of each interior singular vertex is not changed; however, their number is  doubled. 
Let $\bar v\in {\mathcal T'}^{(0)}\cap \partial \Omega '$ be a singular vertex, and denote by ${\bar v}^{0}$ this vertex in the double of $\Omega '$. Note that 
\begin{equation}
\label{eq:sgnindouble}
\mbox{\rm Sgc}_{f}({\bar v}^{0})=2\,\mbox{\rm Sgc}_{f}({\bar v}),
\end{equation}
and therefore, 
\begin{equation}
\label{eq:indexindouble}
\mbox{\rm ind}_{f}({\bar v}^{0})=2\,\mbox{\rm ind}_{f}({\bar v}).
\end{equation}

That is, 
the index of any boundary singular vertex which is not a maximum or a minimum is doubled; however, the number of such vertices is not changed.   In the double, we also have exactly $l$  vertices on which $f$ attains its maximum, and exactly $m$ vertices on which $f$ attains its minimum. It is easy to check that 
\begin{equation}
\label{eq;indexmaxmin}
\mbox{\rm Ind}_{f}({\bar v}^{0})=1, 
\end{equation}
whenever $v$ is a maximum or a minimum vertex. Hence, by applying Theorem~\ref{Th:index} to ${\mathcal G}$ we obtain the following equation.

\begin{equation}
\label{eq:indexwithnoboundary}
2 \sum_{v\in \Omega}\mbox{\rm Ind}_{f}(v)+\sum_{\bar v\in \partial\Omega}2\,\mbox{\rm Ind}_{f}(\bar v)+(l+m)=\chi({\mathcal G})=2\chi(\Omega).
\end{equation}
The assertion of the theorem follows immediately. Let $t$ be the total number of endpoints of the arcs on which $f$ is constant. Since $\frac{l+m}{2}= \frac{t}{4}$ we will often use an equivalent formulation of Equation~(\ref{eq:indexwithnoboundary}).

%\begin{equation}
%\label{eq:indexwithnoboundary}
%\sum_{v\in \Omega}\mbox{\rm Ind}_{f}(v)+\sum_{\bar v\in \partial\Omega}\mbox{\rm Ind}_{f}(\bar v)=\chi(\Omega)- \dfrac{t}{4}. 
%\end{equation} 
\end{proof}

%\begin{Rem}
%\label{re:collapsing}
%{\bf Need to put all of this in the proof!}
%One should observe that after the process of collapsing the cellular decomposition is no longer made of triangles and quadrilaterals. A layer of combinatorial distance one from the $l$ arcs and from the $m$ arcs  is changed. It is now comprises multi-edegs connecting a vertex generated by an arc that collapsed, to all the vertices that connected interior vertices to that vertex.
%\end{Rem}

\section{a few low complexity examples and their surfaces with propellors}
\label{se:low complexity}
In this section we will employ the index lemma (Lemma~\ref{Th:fuctionwithboundary}) and  study a few low complexity examples. While it is possible to analyze some of the examples in this section without applying the index lemma, its usage considerably simplifies  the analysis. These examples pave the way for the understanding of the general case which will be discussed in the next section.

\begin{Exa} (A quadrilateral with two boundary arcs on which $f$ is constant) 
\label{ex:ex1} This example was studied in length in \S\ref{se:Qaud}. However, it is worth noting that in this case $m=1$ and $t=4$. Hence, the right hand-side of Equation~(\ref{eq:indexwithnoboundary}) must be equal to zero.  Since the index of an interior singular vertex is smaller or equal to $-1$, and the index of a singular boundary vertex is smaller or equal to $-\frac{1}{2}$, it follows that in this case, there are no singular vertices. This conclusion is consistent with the assertion of Proposition~\ref{pr:topoflevel}.
\end{Exa}

\begin{Exa} (An annulus with one outer Neumann arc)
\label{ex:ex2}
Let ${\mathcal A}$ be a planar annulus with boundary $\partial{\mathcal A}= E_1\cup E_2$, where $E_1$ is the outer boundary. Let $\alpha_1$ be a closed arc in $E_1$ with endpoints $Q$ and $P$. We solve the DN-BVP as described in the introduction. In particular, we have $m=2$ and $t=2$, and hence that the right hand-side of Equation~(\ref{eq:indexwithnoboundary}) is equal to $-\frac{1}{2}$. Therefore, 
the only possibility is that there exists only one singular boundary vertex which must belong to $E_1\setminus \alpha_1$. We will denote this vertex by $u_{s}$, and its associated level curve by $l_s$. Since the index of $u_s$ is equal to $-\frac{1}{2}$, there must be at least  two arcs of $l_s$  which pass through $u_s$.  It follows by the maximum principle that there are exactly two arcs, and that $l_s$ is simple. Moreover, $E_2\cup l_s$ comprises the boundary of an annulus which we will denote by ${\mathcal A}_{l_s,E_2}$. 
It follows that  ${\mathcal A}$ is topologically the union (along $l_s$) of a topological quadrilateral in which two adjacent vertices have been identified and an annulus.  Such a quadrilateral will henceforth be called a {\it sliced quadrilateral}. The following lemma will show that this decomposition is geometric.
% i.e. that the length of $l_s$ is the same measured with respect to the $l_1$ metric of the induced D-BVP on ${\mathcal A}_{l_s,E_2}$ or with respect to the $l_1$ metric of the 
%induced  DN-BVP in ${\mathcal Q}_{{\mathcal A}_{l_s,E_2}}={\mathcal A} \setminus ({\mathcal A}_{l_s,E_2})^{0}$.

Let ${\mathcal V}=\{u_s, v_1,v_2,\ldots,v_k\}\in l_s$ be the set of vertices enumerated counterclockwise. Recall that some of these vertices are created due to the intersections of edges in ${\mathcal T}^{(1)}$ with $l_s$ (type I), while others may belong to ${\mathcal T}^{(0)}$. For any type I vertex, we define the conductance along the two new edges it induces according to Equation~(\ref{eq:consistent}). In particular, if we let $\bar g$ denote the solution of the DN-BVP on ${\mathcal A}$ which has the same boundary data as $g$, and the same conductance constants on edges which do not have type I vertices,
then $\bar g$ and $g$ have the same values on ${\mathcal T}^{(0)}$ and $\bar g$ is a linear extension of $g$  at vertices of type I.

\begin{Lem}
\label{le:gluingex1} Let $g_2={\bar g}|_{{\mathcal A}_{l_s,E_2}}$, the solution of the  D-BVP defined on ${\mathcal A}_{l_s,E_2}$, and let ${\bar g}_1=g|_{{\mathcal A}_{l_s,E_2}}$, the solution of the  DN-BVP defined on the quadrilateral  ${\mathcal A}\setminus({\mathcal A}_{l_s,E_2})^{0}$. Then 
the length of $l_s$ measured with respect to the flux-gradient metric of $g_1$  is equal to its length measured with respect to the   
flux-gradient metric of $g_2$.
\end{Lem}
\begin{proof}

Since $\bar g$ is harmonic at each vertex in ${\mathcal V}$ which is different from $u_s$, and since the Neumann derivative of $g$ at $u_s$ is zero, we have that
\begin{equation}
\label{eq:equallength}
\sum_{v_i\in {\mathcal V}, v\neq u_s } \sum_{y \sim v_i} c(y,v_i)(\bar g(v_i)-\bar g(y)) +\dfrac{\partial \bar g}{\partial n}({\mathcal A})(u_s)=0.
\end{equation}
We now split the neighbors of each vertex in ${\mathcal V}$ other than $u_s$ into two groups. For 
each $i=1,\ldots, k$, let ${\mathcal A}_{l_s,E_2} (v_i)=\{v_i^1,\ldots, v_i^{j(i)}\}$ be the neighbors of $v_i$ which are contained in  $({\mathcal A}_{l_s,E_2})^{0}$ and let ${\mathcal Q}_{l_s,E_2}(v_i)=\{v_i^{j(i)+1},\ldots, v_i^{t(i)}\}$ be the rest of its neighbors. Let $\{u_s^1,\ldots, u_s^{p}\}$ be the neighbors of $u_s$ in $({\mathcal A}_{l_s,E_2})^{0}$ and let $\{u_s^{p+1},\ldots, u_s^{q}\}$ be the rest of its neighbors. We now rewrite Equation~(\ref{eq:equallength}) in the following form:

\begin{eqnarray}
\label{eq:consistent1}
\sum_{v_i\in {\mathcal V}, v\neq u_s } (\sum_{y \sim v_i \wedge y\in {\mathcal A}_{l_s,E_2}  (v_i)  } c(y,v)(\bar g(v)-\bar g(y))\ \ \  
+\sum_{y \sim v_i \wedge y\in {\mathcal Q}_{{l_s,E_2}(v_i)  }} c(y,v_i)(\bar g(v_i)-\bar g(y))\ \ \ )\  & &  \\
+\ \sum_{r=1}^{p}c(u_s,u_s^r)(g(u_s)-g(u_s^r))+ \sum_{w=p+1}^{q}c(u_s,u_s^w)(g(u_s)-g(u_s^w))\ =&0.&\nonumber
\end{eqnarray}

By the definition of the flux-gradient metric, splitting the sum above into two groups and taking absolute values, the assertion of the lemma follows.

\end{proof}

\begin{figure}[htbp]
 \setlength{\captionindent}{10pt}
    \begin{minipage}[t]{0.5\textwidth}
   \scalebox{.5}{\centering\input{annouter.pstex_t}}
    \caption{A few level curves in Example 4.2.}
    \label{f:annoute}
    \end{minipage}\begin{minipage}[t]{0.5\textwidth}
    \scalebox{.5}{\hskip5mm\centering\input{annoutersur.pstex_t}}
    \caption{The associated surface.}
    \label{f:annoutesurf}
    \end{minipage}
\end{figure}

\end{Exa}

%\begin{multline}
%\label{eq:consistent1}
%\sum_{v_i\in {\mathcal V}, v\neq V_s } \Big(\sum_{y \sim v_i \wedge y\in {\mathcal V}_{ {\mathcal A}_{l_s,E_1}  }(v_i)  } c(y,v)(\bar g(v)-\bar g(y)) \\
%+\sum_{y \sim v_i \wedge y\in {\mathcal Q}_{{\mathcal A}_{l_s,E_1}}(v_i)  } c(y,v_i)(\bar g(v_i)-\bar g(y))\Big)+ \\
%\sum_{r=1}^{p}c(V_s,V_s^r)(g(V_s)-g(V_s^r))+ \sum_{w=p+1}^{q}c(V_s,V_s^w)(g(V_s)-g(V_s^w))=0.
%\end{multline}

%\begin{align}
%\sum_{v_i\in {\mathcal V}, v\neq V_s } &\Big(\sum_{y \sim v_i \wedge y\in {\mathcal V}_{ {\mathcal A}_{l_s,E_1}  }(v_i)  } c(y,v)(\bar g(v)-\bar g(y))\notag\\
%&+\sum_{y \sim v_i \wedge y\in {\mathcal Q}_{{\mathcal A}_{l_s,E_1}}(v_i)  } c(y,v_i)(\bar g(v_i)-\bar g(y))\Big) \label{eq:consistent2}
%\\
%&+\sum_{r=1}^{p}c(V_s,V_s^r)(g(V_s)-g(V_s^r))+ \sum_{w=p+1}^{q}c(V_s,V_s^w)(g(V_s)-g(V_s^w))=0.
%\notag
%\end{align}

%\begin{enumerate}
%\item $g|_{\alpha_1}=k$, for some positive constant $k$, 
%\item $\dfrac{\partial g}{\partial n}|_{(E1\setminus \alpha_1)}=0$, and
%\item $g|_{E_2}=0$.
%\end{enumerate}

\begin{Exa} (An annulus with one inner Neumann arc) 
\label{ex: inner}
The analysis is similar to the one in the previous example. Let ${\mathcal A}$ be a planar annulus with $\partial{\mathcal A}= E_1\cup E_2$, where $E_1$ is the outer boundary. Let $\beta_1$ be a closed arc in $E_2$ with endpoints $Q$ and $P$. We solve the DN-BVP as described in the introduction. In particular, we have $m=2$ and $t=2$, and hence that the right hand-side of Equation~(\ref{eq:indexwithnoboundary}) is equal to $-\frac{1}{2}$. Therefore, 
the only possibility is that there exists only one singular boundary vertex which belongs to $E_2\setminus \beta_1$. We will denote this vertex by $v_{s}$, and its associated level curve by $l_s$. Since the index of $v_s$ is equal to $-\frac{1}{2}$, there must be at least  two arcs of $l_s$  which pass through $v_s$.  It follows by the maximum principle that there are exactly two arcs, and that $l_s$ is simple. Moreover,  $E_1\cup L_s$ comprises the boundary of an annulus which we will denote by ${\mathcal A}_{l_s,E_1}$. 
It follows that  ${\mathcal A}$ is topologically the union (along $l_s$) of a sliced quadrilateral in which two adjacent vertices have been identified to one, $v_s$,  and an annulus.  Arguing in a similar way to Lemma~\ref{le:gluingex1} shows that this decomposition is geometric.  The length of $l_s$ measured with respect to the flux-gradient metric of the induced D-BVP on ${\mathcal A}_{l_s,E_1}$, is the same as measured with respect to the flux-gradient metric of the 
induced  DN-BVP in ${\mathcal Q}_{{\mathcal A}_{l_s,E_2}}={\mathcal A} \setminus ({\mathcal A}_{l_s,E_1})^{0}$.
\end{Exa}

\begin{Rem}
The surface associated with this example is basically obtained by turning the previous one upside down.
\end{Rem}

\begin{Exa} (An annulus with one outer and one inner Neumann arc)
\label{ex:innerandouter}
The analysis of this case relies on the results and principles set forth in the preceding two examples. Let ${\mathcal A}$ be a planar annulus with $\partial{\mathcal A}= E_1\cup E_2$, where $E_1$ is the outer boundary. Let $\alpha_1$ be a closed arc in $E_1$ with endpoints $Q$ and $P$, and let $\beta_1$ be a closed arc in $E_2$ with endpoints $S$ and $T$. We solve the DN-BVP as described in the introduction. In particular, we have $m=2$ and $t=4$, and hence that the right hand-side of Equation~(\ref{eq:indexwithnoboundary}) is equal to $-1$. By using the local structure of an interior singular vertex of index $-1$, the maximum principle or the fact that $g$ has different values on the pairs $\{P,Q\}$ and $\{S,T\}$, one can show that an interior singular vertex of index $-1$ cannot occur. Similarly one rules out the case of a boundary singular vertex of index $-1$. Hence, the only possible way in which equality may hold in Equation~(\ref{eq:indexwithnoboundary}) is the case in which there exist two singular boundary vertices, each of index $-\frac{1}{2}$. With some additional work one shows that  $E_1\setminus\alpha_1$ contains one of these, say $u_{s_{1}}$,  and $E_2\setminus\beta_1$ contains the other, say $u_{s_{2}}$. It follows from  the maximum principle that  $u_{s_1}$ and $u_{s_2}$ have different $g$ values. In particular, their associated level curves $l_{s_1}$ and $l_{s_2}$ are disjoint. As
 in the preceding two examples, there are exactly two arcs (of the appropriate level curve) meeting at a singular vertex. Hence, ${\mathcal A}$ is topologically the union of three pieces. The first is a sliced quadrilateral whose boundary consists of $E_2$ and $l_{s_2}$ which will be denoted by ${\mathcal Q}_{E_2,l_{s_2}}$. The second piece is an annulus whose boundary consists of $l_{s_{2}}$ and $l_{s_1}$ which will be denoted  by ${\mathcal A}_{l_{s_1},l_{s_2}}$.  The third piece is a sliced quadrilateral whose boundary consists of $l_{s_1}$ and $E_1$ which will be denoted by ${\mathcal Q}_{E_1,l_{s_1}}$. We have 
\begin{equation}
\label{eq:annuluswithtwoarcs}
 {\mathcal Q}_{E_2,l_{s_2}}\cap {\mathcal A}_{l_{s_1},l_{s_2}}= l_{s_2}\  \  \mbox{\rm and}\  \ {\mathcal Q}_{E_1,l_{s_1}}\cap {\mathcal A}_{l_{s_1},l_{s_2}}= l_{s_1}.
\end{equation}

A simple generalization of Lemma~\ref{le:gluingex1} shows that the gluing is geometric. That is, for $i=1,2$, the length of $l_{s_i}$ measured with respect to the flux-gradient metric by the induced D-BVP on ${\mathcal A}_{l_{s_1},l_{s_2}}$, equals the length measured with respect to the flux-gradient metric by the induced DN-BVP on ${\mathcal Q}_{E_1,l_{s_1}} $ and ${\mathcal Q}_{E_2,l_{s_2}} $, respectively (as before, one needs to add vertices of type I and type II, if necessary). 
\end{Exa}

\begin{Rem}
The surface associated with this example is  basically the ``union" of the surfaces in the previous two.
\end{Rem}

\begin{Exa} (a planar pair of pants with one outer Nemann arc)
\label{ex:pairwithouter}
Let ${\mathcal P}$ be a planar pair of pants with  $\partial{\mathcal P}=E_1\cup E_2$, where $E_1$ is the outer boundary and $E_2=E_2^1\sqcup E_2^2$ is the inner boundary.
Let $\alpha_1$ be a closed arc in $E_1$ with endpoints $Q$ and $P$. We solve the DN-BVP as described in the introduction. In particular, we have $m=3$ and $t=2$, and hence that the right hand-side of Equation~(\ref{eq:indexwithnoboundary}) is equal to $-\frac{3}{2}$. By applying the maximum principle, the case in which there exist three singular vertices each having its index equal to $-\frac{1}{2}$, and the case of two singular boundary vertices, one of which has index $-1$ and the other has index $-\frac{1}{2}$, can be easily ruled out. 
We are left with the possibility that there exists one interior singular vertex, $u_{s_1}$, whose index equals $-1$ and one boundary singular vertex, $u_{s_2}$, whose index is equal to $-\frac{1}{2}$. Let $l_{s_1}$ be the singular level curve passing through $u_{s_1}$ and let $l_{s_2}$ the singular level curve passing through $u_{s_2}$. Arguing as in the previous examples, there are exactly two arcs of the singular curve passing through $u_{s_2}$ and four arcs passing through $u_{s_1}$.

There are two cases to consider. First assume that $u_{s_1}$ and $u_{s_2}$ have the same $g$-value. This implies (by using the maximum principle) that they lie on the same singular level curve, which we will denote by $l_s$. The second case occurs when $u_{s_1}$ and $u_{s_2}$ have different $g$-values.  In the first case the topological decomposition of ${\mathcal P}$ is the following. The two annuli ${\mathcal A}_{E_2^1, l_s}$ and ${\mathcal A}_{E_2^2, l_s}$ which intersect at $u_{s_1}$ are attached to the sliced quadrilateral 
${\mathcal Q}_{l_{s}, E_1}$ along the union of their boundaries, $l_{s}$. Observe that in this case ${\mathcal Q}_{l_{s}, E_1}$ has one singular boundary arc, $l_s$, at $u_{s_1}$.

\begin{figure}[htbp]
 \setlength{\captionindent}{20pt}
    \begin{minipage}[t]{0.5\textwidth}
   \scalebox{.3}{\centering\input{pairoutfirst.pstex_t}}
    \caption{The first case of Example 4.13.}
    \label{f:annoutea}
    \end{minipage}\begin{minipage}[t]{0.5\textwidth}
    \scalebox{.3}{\centering\input{pairoutfirstsur.pstex_t}}
    \caption{The associated surface.}
    \label{f:annoutesurfa}
    \end{minipage}
\end{figure}

In the second case the topological decomposition of ${\mathcal P}$ is the following.
The two annuli ${\mathcal A}_{E_2^1,  l_{s_1}}$ and ${\mathcal A}_{E_2^2,  l_{s_2}}$ which intersect at $u_{s_1}$ are attached to the (singular) annulus
${\mathcal A}_{l_{s_1}, l_{s_2}}$ along their common boundary, $l_{s_1}$; the annulus ${\mathcal A}_{l_{s_1}, l_{s_2}}$ is attached to the sliced quadrilateral ${\mathcal Q}_{l_{s_2}, E_1}$ via their common boundary $l_{s_2}$.  It can be shown by a generalization of Lemma~\ref{le:gluingex1}, that the gluing is geometric.

\begin{figure}[htbp]
 \setlength{\captionindent}{20pt}
    \begin{minipage}[t]{0.5\textwidth}
   \scalebox{.3}{\centering\input{pairoutsec.pstex_t}}
    \caption{The second case of Example 4.13.}
    \label{f:annoutec}
    \end{minipage}\begin{minipage}[t]{0.45\textwidth}
    \scalebox{.3}{\centering\input{pairoutsur.pstex_t}}
    \caption{The associated surface.}
    \label{f:annoutesurfc}
    \end{minipage}
\end{figure}

\end{Exa}

We finish this section with one more example which illustrates some of the combinatorial complexity of  higher genus cases. We will not provide a complete analysis of this case and leave the completion of the details to the reader.

\begin{Exa} (a planar pair of pants with two outer Nemann arcs)
\label{ex:pairwithtwoouter}
Let ${\mathcal P}$ be a planar pair of pants with its boundary ${\mathcal P}=E_1\sqcup E_2$, where $E_1$ is the outer boundary and where $E_2=E_2^1\sqcup E_2^2$ is the inner boundary.
Let $\alpha_1$ be a closed arc in $E_1$ with endpoints $P_1$ and $Q_1$ and let $\alpha_2$ be another closed arc in $E_1$ with endpoints $P_2$ and  $Q_2$, respectively; further assume that  $P_1, Q_1, P_2$ and $Q_2$ are ordered counterclockwise. 
We solve the DN-BVP as described in the introduction. In particular, we have $m=3$ and $t=4$, and hence that the right hand-side of Equation~(\ref{eq:indexwithnoboundary}) is equal to $-2$. 
There are several cases in which two boundary vertices, the index of each is equal to $-\frac{1}{2}$, and an interior singular vertex of index which is equal to $-1$ may occur, and we will now describe two of these. 

Let $u_{s_1}\in  Q_1 P_2$ and $u_{s_2}\in Q_2 P_1$ be the two singular boundary vertices and let $u_b$ be the interior singular boundary vertex.
First assume that $u_{s_1}$ and $u_{s_2}$ attain the same $g$ values and belong to the same level curve which is different from the $g$ value attained at $u_b$. It  follows that $l_s$, this singular level curve, is a closed (piecewise linear) curve. It follows by the maximum principle that $E_2$ is contained in the domain bounded by $l_s$. Also, $l_b$, the singular level curve which passes through $u_b$ is a piecewise figure eight curve, and (necessarily) the $g$-value of $u_b$ is smaller than that of the $g$-value of $u_{s_1}$.  Hence, in this case ${\mathcal P}$ has the following topological decomposition. A quadrilateral  ${\mathcal Q}_{\mbox{\rm {\tiny right}}}$ which has  $u_{s_1},Q_1,P_1$ and $u_{s_2}$ as its corners. A quadrilateral  ${\mathcal Q}_{\mbox{\rm {\tiny left}}}$ which has  $u_{s_1},P_2,Q_2$ and $u_{s_2}$ as its corners. A singular annulus ${\mathcal A}_{l_s,l_b}$ with its  singular boundary curve being $l_b$, the other boundary curve being $l_s$.  The singular annulus ${\mathcal A}_{l_s,l_b}$ is attached to 
${\mathcal Q}_{\mbox{\rm {\tiny right}}}$ along the right arc of  $l_s$, the one which  connects $u_{s_1}$ to $u_{s_2}$, and to ${\mathcal Q}_{\mbox{\rm {\tiny right}}}$ along the left arc of  $l_s$ which connects $u_{s_1}$ to $u_{s_2}$.

\begin{figure}[htbp]
 \setlength{\captionindent}{20pt}
    \begin{minipage}[t]{0.5\textwidth}
   \scalebox{.5}{\centering\input{pairtwoout.pstex_t}}
    \caption{The first case of Example 4.13.}
    \label{f:annouted}
    \end{minipage}\begin{minipage}[t]{0.45\textwidth}
    \scalebox{.45}{\centering\input{pairtwooutfirstsur.pstex_t}}
    \caption{The associated surface.}
    \label{f:annoutesurfd}
    \end{minipage}
\end{figure}

The singular curve $l_b=l_{\mbox{\rm {\tiny left}}}\cup u_b\cup l_{\mbox{\rm {\tiny right}}}$ bounds two annuli, ${\mathcal A}_{E_2^1,l_{\mbox{\rm {\tiny left}}}}$ which has as its boundary $E_2^1$ and $u_b\cup l_{\mbox{\rm {\tiny left}}}$, and ${\mathcal A}_{E_2^2,l_{\mbox{\rm {\tiny right}}}}$ which has as its boundary $E_2^2$ and $u_b\cup l_{\mbox{\rm {\tiny right}}}$. These two annuli intersect (only) at the vertex $u_b$. 

%The lower arc of $l_b$ connecting $u_{s_1}$ to $u_{s_2}$ forms one of the boundary arcs of a quadrilateral ${\mathcal Q}_{\mbox{\rm{\tiny bottom}}}$ which has  $u_{s_2},Q_2,P_1$ and $u_{s_1}$ as its corners. 

We now handle the case in which $u_{s_1}$ and $u_{s_2}$ have the same $g$-values but belong to different level curves. 
 Let $l_{s_1}$ be the singular level curve passing through $u_{s_1}$, and let $l_{s_2}$ the singular level curve passing through $u_{s_2}$. It is easy to check that $l_{b}$ must intersect $Q_1 P_2$ in two points which we will denote by $S_1$ and $S_2$, respectively, with $S_1$ between $u_{s_1}$ and $Q_1$ and with $S_2$ between $u_{s_1}$ and $P_2$. Similarly, let $T_1$ be the intersection point of $l_b$ with $ Q_2 P_1$ which is between $u_{s_2}$ and $P_1$, and let $T_2$ be the intersection point of $l_b$ with $ Q_2 P_1$ which is between $u_{s_2}$ and $Q_2$.
 Also, $l_{s_1}$ is simple and closed, $l_{s_1}\cap P_2 Q_1=\{u_{s_1}\}$, and the region it bounds contains $E_2^1$. Symmetrically, $l_{s_2}$ is simple and closed, $l_{s_2}\cap Q_2 P_1=\{u_{s_2}\}$, and the region it bounds contains $E_2^2$.
 
In this case, the topological decomposition of ${\mathcal P}$ is the following.  A quadrilateral  ${\mathcal Q}_{\mbox{\rm {\tiny left}}}$ which has  $S_2,P_2,Q_2$ and $T_2$ as its corners. A quadrilateral  ${\mathcal Q}_{\mbox{\rm {\tiny right}}}$ which has  
 $S_1,Q_1,P_1$ and $T_1$ as its corners. A sliced quadrilateral 
${\mathcal Q}_{\mbox{\rm {\tiny top}}}$ which has  $S_2,S_1$ and $u_{s_1}$ as its corners; it is attached to ${\mathcal Q}_{\mbox{\rm {\tiny left}}}$ along the arc of $l_b$ determined by $S_2$ and $u_b$, and to ${\mathcal Q}_{\mbox{\rm {\tiny right}}}$ along the arc of $l_b$ determined by $S_1$ and $u_b$. A sliced quadrilateral 
${\mathcal Q}_{\mbox{\rm {\tiny bottom}}}$ which has  $T_1,T_2$ and $u_{s_2}$ as its corners; it is attached to ${\mathcal Q}_{\mbox{\rm {\tiny left}}}$ along the arc of $l_b$ connecting $T_2$ and $u_b$, and to ${\mathcal Q}_{\mbox{\rm {\tiny right}}}$ along the arc of $l_b$ connecting $T_1$ and $u_b$. The last two pieces are two annuli, ${\mathcal A}_{E_2^1,l_{s_1}}$ and ${\mathcal A}_{E_2^2,l_{s_2}}$ that are attached to the above sliced  quadrilaterals along $l_{s_1}$ and $l_{s_2}$, respectively. The two sliced quadrilaterals intersects (only) at the vertex $u_b$.
As before, an extension of Lemma~\ref{le:gluingex1} shows that in both cases  the gluing is geometric.

\begin{figure}[htbp]
 \setlength{\captionindent}{20pt}
    \begin{minipage}[t]{0.5\textwidth}
   \scalebox{.4}{\centering\input{pairtwooutsecond.pstex_t}}
    \caption{The second case of Example 4.13.}
    \label{f:annoutee}
    \end{minipage}\begin{minipage}[t]{0.45\textwidth}
    \scalebox{.4}{\centering\input{pairtwooutsecondsur.pstex_t}}
    \caption{The associated surface.}
    \label{f:annoutesurfe}
    \end{minipage}
\end{figure}

\end{Exa}

%The sliced boundary arc of ${\mathcal Q}_1$ which is the singular level curve $l_{s_2}$, is also the outer boundary component of a singular annulus ${\mathcal A}_{l_b,l_{s_2}}$ that has $l_b$ as its second and singular boundary component.   Finally, as in the previous case, we have two annuli ${\mathcal A}_{E_2^1,l_{\mbox{\rm {\tiny left}}}}$ and ${\mathcal A}_{E_2^2,l_{\mbox{\rm {\tiny right}}}}$ meeting at the vertex $V_b$. 

\newpage
\section{the general case - an $m$-connected bounded planar region, $m>2$}
\label{se:main}

\noindent{\bf Proof of Theorem~\ref{Th:ladder}.} 
Let $\{0, p_1,p_2,\ldots,p_{n-1},k\}$ be the set of values of 
$g$ at the singular vertices arranged in an increasing order.
We first construct a topological decomposition of $\Omega$.
 For $i=0,\ldots, k$, consider the sub-domain of $\Omega$ defined by 
\begin{equation}
\label{eq:sub}
\Omega_i=\{x\in\Omega\  |\  p_i<g(x)<p_{i+1}\},
\end{equation}
(where the value at $x$ which is not a vertex is defined by the affine extension of $g$). In general $\Omega_i$ is multi-connected and (by definition) contains no singular vertices in its interior. Let $g_i=g|_{\Omega_i}$ be the restriction of $g$ to $\Omega_i\cup\partial  \Omega_i$. The definition of $g_i$ involves (as in the proof of Theorem~\ref{Th:rectangle}) introducing new vertices (of type I and type II), and  new edges and their conductance constants.  In particular,  each $g_i$ is the solution of a D-BVP or a DN-BVP, on each one of the components of $\Omega_i$.

By applying Equation~(\ref{eq:indexwithnoboundary}) in the proof of Lemma~\ref{Th:fuctionwithboundary} to $g_i$ and each component of $\Omega_i$ whose boundary is a Jordan curve which contains no singular vertices, we obtain that there are two cases to consider.  First, a component  of $\Omega_i$  is simply connected and therefore $m=1$ and $t=4$, hence it is a quadrilateral. Second,  a component of $\Omega_i$ is not simply-connected. In this case we must have $m=2$ and $t=0$, hence this component must be an annulus. 

We now treat the remaining cases. First assume that the boundary of a component $\Omega_i^{\mbox {\rm {\tiny Jordan}}}$ of $\Omega_i$ is a Jordan curve and contains at least one singular vertex which we will denote by $v_s$. It follows that the value of $g$ at $v_s$ is either $p_i$ or
$p_{i+1}$, and that it does not belong to 
$\alpha_1\cup\ldots \cup\alpha_l\cup  (E_2\setminus( \beta_1\cup\ldots\beta_m))$. 

According to the index of $v_s\in \Omega$ with respect to $g$, we now replace a small neighborhood of $v_s$ in $\Omega$ by several disjoint piecewise linear wedges. Each wedge has a copy of $v_s$ as a single vertex, and two consecutive arcs of the associated singular level curve, that are contained in the neighborhood, meeting at $v_s$. If $v_s\in \partial \Omega$, then the only difference from the above, is that exactly two of the wedges will contain as one their arcs part of $\partial \Omega$.

\begin{figure}[htbp]
\begin{center}
 \scalebox{.40}{ \input{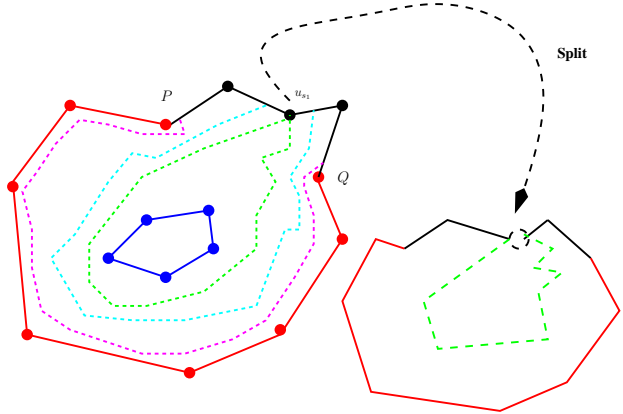}}
 \caption{Splitting at $u_{s_1}$.}
\label{figure:split}
\end{center}
\end{figure}

It follows that after finitely many steps, the boundary of $\Omega_i^{\mbox{\rm \tiny{ Jordan}}}$ is turned into a Jordan domain with no singular vertices on it. Hence Lemma~\ref{Th:fuctionwithboundary} may be applied to the induced Jordan domain and allows us to deduce that $\Omega_i^{\mbox{\rm {\tiny Jordan}}}$ is either an annulus or a quadrilateral. It now follows that before the splitting at the singular vertices occurred,  $\Omega_i^{\mbox{\tiny {\rm Jordan}}}$ was either a quadrilateral in which two adjacent vertices have been identified or an annulus. One should note that singular vertices along the level curves are basically ``straightened" along this process in such a way that they become non-singular viewed from $g_i$ and $\Omega_i^{\mbox{\tiny {\rm Jordan}}}$ with this boundary modified to a Jordan curve.   

We must also consider the case in which the boundary of a component fails to be a Jordan curve. As Example~\ref{ex:ex2} shows, this may already occur in the case of $\Omega$ being an  annulus. This case is treated similarly to the previous one we discussed above (see also Figure~\ref{figure:split}). 

Thus, we conclude that we may decompose  $\Omega$ into a union (with disjoint interiors) of annuli, quadrilaterals or sliced quadrilaterals. We continue the proof by showing that the gluing is geometric, i.e. that with respect to the boundary value problems induced on each of the components, the common boundary has the same  length measure with respect to the flux-gradient metric.

\begin{Lem}
\label{le:gluingex2}
Let $L$ be a connected arc which is contained in $\Gamma_i\cap \Upsilon_j$, where $\Gamma_i$ is a component of $\Omega_i$ and $\Upsilon_j$ is a component of $\Omega_j$, for some $i$ and $j$. Then, the length of $L$ measured with respect to the flux-gradient metric induced by $g_i|_{\Gamma_i}$ is equal to its length measured with respect to the flux-gradient metric induced by $g_j|_{\Upsilon_j}$.
\end{Lem}
\dem
The proof is a direct generalization of Lemma~\ref{le:gluingex1} and follows by applying the index lemma (Lemma~\ref{Th:fuctionwithboundary}) to rule out several cases. Hence, we will only give the details in a few cases. 

Observe that the cases in which $L$ is a Neumann arc, or contains a Neumann arc are clearly not possible. By applying the index lemma it can be shown that the cases in which, both $\Gamma_i$ and $\Upsilon_j$ are both quadrilaterals is not possible unless $L$ is, without loss of generality,  the right boundary of $\Gamma_i$ as well as the left boundary of  $\Upsilon_j$ (this means that $j=i+1$). One then uses the fact that 
the induced D-BVP on $\Gamma_i\cup \Upsilon_{i+1}$ is harmonic on $L$ to deduce the assertion. One treats the case in which both $\Gamma_i$ and $\Upsilon_j$ are annuli in a similar way; deducing that, without loss of generality, $j=i+1$, and that the outer boundary component of $\Gamma_i$ is equal to the boundary component of $\Upsilon_{i+1}$ which corresponds to the $g$ value $i+1$. Again, one uses the harmonicity of the induced D-BVP solution (defined on $\Gamma_i\cup\Upsilon_j$) on $L$ to obtain the assertion.

\begin{figure}[htbp]
\begin{center}
 \scalebox{.4}{ \input{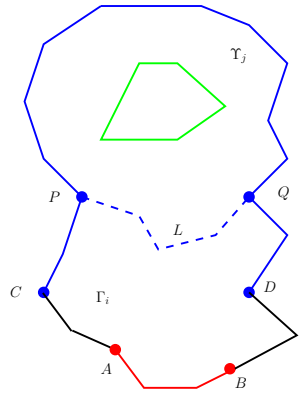}}
 \caption{Viewing $\Gamma_i$ and $\Upsilon_j$.}
\label{figure:gluing}
\end{center}
\end{figure}

Assume that  (without loss of generality) $\Gamma_i$ is a quadrilateral and that $\Upsilon_j$ is an annulus. Assume that $L$ is contained in the component of $\partial \Upsilon_j$, denoted by $\partial \Upsilon_j^{p_{i+1}}$, which corresponds to the $g$ value $p_{i+1}$. Let $\partial \Upsilon_j^{p_{i}}$ be the second component of $\Upsilon_j$ (which corresponds to the $g$-value $p_i$), and  let $L_{i+1}^j=  \partial \Upsilon_j^{p_{i+1}}\setminus L$. Let $P$ and $Q$ be the endpoints of $L$. Assume that $P$ and $Q$ correspond to the $g$ value $p_{i+1}$ and are singular vertices. Since $L$ must lie on $\partial \Upsilon_j$, it follows that the domain bounded by  $\Gamma_i\cup\Upsilon_j\setminus L$ is an annulus. However, since in this case $t=2$, we much have (by the index lemma)  a boundary singular vertex, this is absurd.

We now treat one case in which the index lemma does not provide an obstruction for an intersection (see \S\ref{se:low complexity} for more). The setting is as in the above case, with $\Gamma_i$ being this time a sliced quadrilateral. 
Let $,A,C,D,B$ be the vertices of $\Gamma_i$ arranged clockwise and let $A,B$ be the vertices which are identified. The value of $g_i$ on the arc $AB$ is $p_{i+1}$.

\begin{figure}[htbp]
\begin{center}
 \scalebox{.4}{ \input{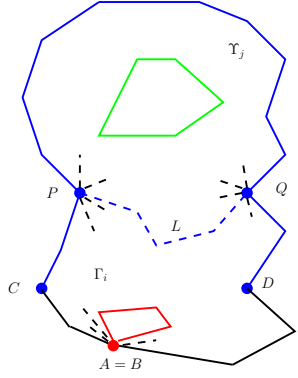}}
 \caption{Viewing $\Gamma_i$ and $\Upsilon_j$.}
\label{figure:gluing1}
\end{center}
\end{figure}

We need to prove that 
\begin{equation}
\label{eq:fluxthroughL}
\sum_{x\in L} \frac{\partial g_i}{\partial n}(\Gamma_i)(x) +\sum_{x\in L}\frac{\partial g_j}{\partial n}(\Upsilon_j)(x)=0.
\end{equation}

By applying Green's theorem to $\Gamma_i, \Upsilon_j$ and $\Gamma_i\cup\Upsilon_j$, respectively, we obtain the following equations:

\begin{eqnarray} 
0 & = &\sum_{x\in BA}\frac{\partial g_i}{\partial n}(\Gamma_i)(x)  +\sum_{x\in CP}\frac{\partial g_i}{\partial n}(\Gamma_i)(x) +\sum_{x\in L}\frac{\partial g_i}{\partial n}(\Gamma_i)(x) +\sum_{x\in QD}\frac{\partial g_i}{\partial n}(\Gamma_i)(x),  \\
    0 & =&  \sum_{x\in \partial \Upsilon_j^{p_{i}}}\frac{\partial g_j}{\partial n}(\Upsilon_j)(x) +
    \sum_{x\in L_{i+1}^j}\frac{\partial g_j}{\partial n}(\Upsilon_j)(x) +\sum_{x\in L}\frac{\partial g_j}{\partial n}(\Upsilon_j)(x),     \\ 
    0& = &\sum_{x\in BA}\frac{\partial g_i}{\partial n}(\Gamma_i)(x)   +\sum_{x\in CP}\frac{\partial g_i}{\partial n}(\Gamma_i)(x) +\sum_{x\in L_{i+1}^j}\frac{\partial g_j}{\partial n}(\Upsilon_j)(x) + \\ \nonumber& &  \sum_{x\in QD}\frac{\partial g_i}{\partial n}(\Gamma_i)(x) +
     \sum_{x\in \partial \Upsilon_j^{p_i}}\frac{\partial g_j}{\partial n}(\Upsilon_j)(x).  
\end{eqnarray}

By subtracting the third equation from the first and adding the second equation, Equation~(\ref{eq:fluxthroughL}) follows. One verifies by using the method above,  that other (finitely many) possible cases, lead as well to assertion of the Proposition.
\eop{Lemma~\ref{le:gluingex2}}

\vskip.1in

We now successively apply the assertions of Theorem~\ref{Th:rectangle} and Theorem~\ref{Th:annulus} to the appropriate components of $\Omega$. One needs only observe that the tilling thus obtained is consistent as defined in the discussion  preceding Proposition~\ref{pr:consistent} (see also \cite{Her1}).
This follows by a straightforward generalization of the arguments given in the proof of 
Proposition~\ref{pr:consistent} (see the proof of Theorem~\ref{Th:annulus} and the proofs of \cite[Theorem 0.1, Theorem 0.4]{Her1}).
\eop{Theorem~\ref{Th:ladder}}

%\begin{Rem}
%\label{re:inductionwilldo}
%One may also prove Theorem~\ref{Th:ladder} by an induction on the number of boundary components. However, the assertion of Proposition~\ref{pr:onethatenclose} must be used as well as an extension of $g_0$ to $g_1$ over the singular curve $L_1$. Theorem~\ref{le:length of pl} needs to be used in order to prove equality of the $l_1$-length of $L_1$ according to both the $g_0$ and the $g_1$ metric. Overall, we found the proof which does not use induction conceptually more gratifying. We could also stop the process once a planar pair of pants is encountered, thus using Theorem~\ref{Th:pair} directly. Finally, Theorem~\ref{Th:pair} is of course a special case of the theorem above. Still, we maintain that this special case deserves its own proof. 
%\end{Rem}

\bigskip

\begin{Rem}
The analysis of the cone singularities is almost identical to the one carried in \cite[Subsection 4.2]{Her1}. One observes the following additional cases. 
The presence  of propellors results  in the creation of new cone singularities of angle $\pi/2$ at each vertex. Hence, under the doubling, if such vertex belongs to a unique propellor, the cone angle will change to $\pi$ (see for instance vertices $P_1,Q_1,P_2,Q_2$ in Example~\ref{ex:pairwithtwoouter}). A similar analysis holds if such a vertex belongs to two rectangles and a non-singular component of a Euclidean cylinder (yielding a cone angle of $4 \pi$ in the double). Finally, at the singular vertex of a sliced rectangle the cone angle is $\pi$, and the analysis of the changes of this angle under doubling is easy to carry (see for instance vertex $v_s$ in Example~\ref{ex:ex2}).
\end{Rem}

\medskip

\begin{Rem}
\label{re:valuesofg}
There is a technical difficulty in our construction if some pair of adjacent vertices of ${\mathcal T}^{(0)}$ has the same $g$-value (the first occurrence is in Equation~(\ref{eq:index})).  One may generalize the definitions and the index formula to allow rectangles of area zero, as one solution. 
For a discussion of this approach and others see \cite[Section 5]{Ke}. Experimental evidence shows that when the cell decomposition is complicated enough, even when the conductance function is identically equal to $1$ and the cells are triangles, such equality rarely happens (for D-BVP).
\end{Rem}

\medskip

\begin{Rem}
\label{re:embedd}
The existence of singular curves for $g$ results in the fact that some rectangles are not embedded in the target. This is evident by Remark~\ref{re:singular} and the proof of  Theorem~\ref{Th:ladder}. Since some of the cylinders or sliced quadrilaterals constructed have a singular boundary component, it is clear that some points in different rectangles that lie on this level curve will map to the same point. However, this occurs only in the situation described above, and since this fact is not of essential interest to us, we will not go into more details. 
\end{Rem}

\end{document}